\documentclass[12pt, reqno]{amsart}
\usepackage{amsmath, amsthm, amscd, amsfonts, amssymb, graphicx, color}
\usepackage[bookmarksnumbered, colorlinks, plainpages]{hyperref}
\hypersetup{colorlinks=true,linkcolor=red, anchorcolor=green, citecolor=cyan, urlcolor=red, filecolor=magenta, pdftoolbar=true}

\usepackage{enumerate}
\usepackage{mathrsfs}

\newtheorem{theorem}{Theorem}[section]
\newtheorem{lemma}[theorem]{Lemma}
\newtheorem{assumption}[theorem]{Assumption}

\newtheorem{corollary}[theorem]{Corollary}
\theoremstyle{definition}
\newtheorem{definition}[theorem]{Definition}

\theoremstyle{remark}
\newtheorem{remark}[theorem]{Remark}
\numberwithin{equation}{section}

\begin{document}

\setcounter{page}{1}

\title[local Hardy spaces]{Local Hardy spaces associated with ball quasi-Banach function spaces and their dual spaces}

\author[X. Chen]{Xinyu Chen}
\address{Xinyu Chen, School of Science, Nanjing University of Posts and Telecommunications, Nanjing 210023, China}
\email{\textcolor[rgb]{0.00,0.00,0.84}{chenxinyu1130@163.com}}

\author[J. Tan]{Jian Tan*}
\address{Jian Tan(Corresponding author), School of Science, Nanjing University of Posts and Telecommunications, Nanjing 210023, China}
\email{\textcolor[rgb]{0.00,0.00,0.84}{tj@njupt.edu.cn; tanjian89@126.com}}





\begin{abstract}
Let $X$ be a ball quasi-Banach function space on $\mathbb R^{n}$ and $h_{X}(\mathbb R^{n})$ the local Hardy space associated with $X$. In this paper, under some reasonable assumptions on $X$, the infinite and finite atomic decompositions for the local Hardy space $h_{X}(\mathbb R^{n})$ are established directly, without relying on the relation between $H_{X}(\mathbb R^{n})$ and $h_{X}(\mathbb R^{n})$. Moreover, we apply the finite atomic decomposition to obtain the dual space of the local Hardy space $h_{X}(\mathbb R^{n})$. Especially, the above results can be applied to several specific ball quasi-Banach function spaces, demonstrating their wide range of applications.
\end{abstract} \maketitle

\section{Introduction}
The Hardy space $H^{p}(\mathbb R^{n})$ with $p\in(0,1)$ is frequently employed as a suitable substitute of the Lebesgue space $L^{p}(\mathbb R^{n})$, particularly in the context of studying the boundedness of operators. Consequently, Hardy spaces play a pivotal role in various fields of analysis and partial differential equations. Recall that the classical Hardy space was originally introduced by Stein and Weiss \cite{1960On} and subsequently refined and expanded upon in a systematic manner by Fefferman and Stein \cite{fefferman1972vari}. In recent times, various variants of classical Hardy spaces have been introduced and their real-variable theories have been thoroughly studied. These variants encompass weighted Hardy spaces, (weighted) Hardy-Morrey spaces, Hardy-Orlicz spaces, Lorentz Hardy spaces. It is noteworthy that the elementary spaces on which the aforementioned Hardy-type spaces were built, including weighted Lebesgue spaces, (weighted) Morrey spaces, Orlicz spaces and Lorentz spaces, are not necessarily quasi-Banach function spaces, which implies the restriction of the notion of (quasi-)Banach function spaces and provides a motivation to establish a unified theory for the aforementioned Hardy-type spaces. Therefore, Sawano et al. \cite{sawano2017hardy} introduced the ball quasi-Banach function space $X$, the related Hardy space $H_{X}(\mathbb R^{n})$ and the related local Hardy space $h_{X}(\mathbb R^{n})$ so that the above spaces fall into this generalized framework. Furthermore, Sawano et al. \cite{sawano2017hardy} established various maximal function characterizations of $H_{X}(\mathbb R^{n})$ and $h_{X}(\mathbb R^{n})$ by assuming the boundedness of the Hardy--Littlewood maximal operator on the $p$-convexification of $X$. On the other hand, the atomic decomposition significantly contributes to the investigation of the boundedness of operators on Hardy-type spaces. In fact, Coifman \cite{coifman1974real} pioneered the atomic decomposition characterization of Hardy spaces on $\mathbb {R}$, with Latter \cite{1978A} extending this characterization to higher dimensions. Sawano et al. \cite{sawano2017hardy} also demonstrated that atomic characterizations of Hardy-type spaces associated with a ball quasi-Banach function $X$ strongly depend on the Fefferman--Stein vector-valued maximal inequality and the boundedness of the powered Hardy--Littlewood maximal operator on the associate space of $X$. We refer the reader to \cite{tanzhang,wang2020applications,wang2021weak,zhang2021weak} for more details on Hardy-type spaces associated with ball quasi-Banach function spaces.\par
The dual theory of classical Hardy spaces on $\mathbb R^{n}$ plays a significant role in many branches of analysis and has undergone systematic development to date.
Recall that John and Nirenberg \cite{john1961functions} introduced the bounded mean oscillation function space BMO$(\mathbb R^{n})$, which was proved to be the dual space of the Hardy space $H^{1}(\mathbb R^{n})$ by Fefferman and Stein \cite{fefferman1972vari}. To extend this result to the classical Hardy space $H^{p}(\mathbb R^{n})$ for any given $p\in(0,1]$, Campanato \cite{campanato1964proprieta} introduced the Campanato space $\mathcal{C}_{\alpha,q,s}(\mathbb R^{n})$, which coincides with BMO$(\mathbb R^{n})$ when $\alpha=0$. Thus, Campanato proved that for any $q\in[1,\infty]$ and any integer $s\in[0,\infty)\cap[n(1/p-1),\infty)$, the dual space of $H^{p}(\mathbb R^{n})$ is $\mathcal{C}_{1/p-1,q,s}(\mathbb R^{n})$. It is worth pointing out that Yan et al.\cite{yan2020intrinsic} introduced the Campanato-type space function space $\mathcal{L}_{X,q,s}(\mathbb R^{n})$ associated with $X$, which was proved to be the dual space of $H_{X}(\mathbb R^{n})$ under the condition that $X$ is concave. Very recently, Zhang et al. \cite{zhang2022new} introduced a new ball Campanato-type function space $\mathcal{L}_{X,q,s,d}(\mathbb R^{n})$ associated with the ball quasi-Banach function space $X$ and improved the aforementioned dual result by eliminating the requirement for $X$ to be concave.  

The main target of this paper is to obtain the infinite and finite atomic decompositions for the local Hardy space $h_{X}(\mathbb R^{n})$ associated with the ball quasi-Banach function $X$. In fact, the authors give the infinite and finite decompositions for the global Hardy space $H_{X}(\mathbb R^{n})$ in \cite[Subsection 3.6]{sawano2017hardy} and \cite[Theorem 1.10]{yan2020intrinsic}, respectively. Moreover, Wang et al. \cite[Theorem 4.8]{wang2020applications} showed the infinite decomposition for the local Hardy space $h_{X}(\mathbb R^{n})$ via applying the relation between $H_{X}(\mathbb R^{n})$ and $h_{X}(\mathbb R^{n})$. Inspired by \cite{tan2023real}, we establish the infinite and finite atomic decompositions for the local Hardy space $h_{X}(\mathbb R^{n})$ by avoiding using the relation between $H_{X}(\mathbb R^{n})$ and $h_{X}(\mathbb R^{n})$. As an application of the above finite atomic decomposition, we shall obtain the dual space of the local Hardy space $h_{X}(\mathbb R^{n})$. Finally, we apply the main theorems to several concrete ball quasi-Banach function spaces, which implies that the main results have a wide range of generality.\par\par
The paper is organized as follows. In Section~\ref{se2}, we present some necessary definitions and important lemmas. In Section~\ref{se3}, we present the maximal function characterizations for the local Hardy space $h_{X}(\mathbb R^{n})$, which contribute to the proof of atomic decompositions in Section~\ref{se4}. In Section~\ref{se4}, we obtain the infinite and finite atomic decompositions of $h_{X}(\mathbb R^{n})$ under some reasonable assumptions on a ball quasi-Banach function space $X$. In Section~\ref{se5}, we employ the finite atomic decomposition theorem to derive the dual space of $h_{X}(\mathbb R^{{n}})$. Finally, in Section~\ref{se6}, we apply the main results obtained in aforementioned sections to some concrete ball quasi-Banach function spaces.\par
Throughout this paper, $C$ or $c$ denotes a positive constant that is independent of the main parameters involved but may vary at each occurrence. To denote the dependence of the constants on some parameter $s$, we will write $C_{s}$. We denote $f\leq Cg$ by $f\lesssim g$. If $f\lesssim g\lesssim f$, we write $f\sim g$ or $f\approx g$. Denote $Q(x,l(Q))$ the closed cube centered at $x$ and of side-length $l(Q)$. Similarly, given $Q=Q(x,l(Q))$ and $\lambda>0$, $\lambda Q$ means the cube with the same center $x$ and with side-length $\lambda l(Q)$. We denote $Q^{*}=2\sqrt{n}Q$. Let $\mathbb N:=\{1,2,\cdots\}$, $\mathbb Z_{+}:=\mathbb N\cup\{0\}$ and $\mathbb Z^{n}_{+}:=(\mathbb Z_{+})^{n}$. We use $\mathscr{M}(\mathbb R^{n})$ to denote the set of all measurable functions on $\mathbb R^{n}$.

\section{Preliminaries}\label{se2}
In this section, we present some concepts and lemmas that will be used throughout the paper.\par
First, we recall the definitions of ball quasi-Banach function spaces and their related local Hardy spaces $h_{X}(\mathbb R^{n})$. For any $x\in\mathbb R^{n}$ and $r\in(0,\infty)$, let $B(x,r):=\{y\in\mathbb R^{n}\colon\vert x-y\vert<r\}$ and
\begin{equation}\label{eq2.1}
\mathbb B(\mathbb R^{n}):=\left\{B(x,r)\colon x\in\mathbb R^{n}\ {\rm and}\ r\in(0,\infty)\right\}.
\end{equation}\par
The concept of ball quasi-Banach function spaces on $\mathbb R^{n}$ is as follows. For more details, see \cite{sawano2017hardy,tan2024}.
\begin{definition}\label{def2.1}
    Let $X\subset\mathscr{M}(\mathbb R^{n})$ be a quasi-normed linear space equipped with a quasi-norm $\| \cdot \|_{X}$ which makes sense for all measurable functions on $\mathbb R^{n}$. Then $X$ is called a \emph{ball\ quasi\text{-}Banach\ function\ space} on $\mathbb R^{n}$ if it satisfies
    \begin{enumerate}[\quad(i)]
    \item if $f\in\mathscr{M}(\mathbb R^{n})$, then $\|f\|_{X}=0$ implies that $f=0$ almost everywhere;
    \item if $f, g\in\mathscr{M}(\mathbb R^{n})$, then $|g|\leq|f|$ almost everywhere implies that $\|g\|_{X}\leq\|f\|_{X}$;
    \item if $\{f_{m}\}_{m\in\mathbb N}\subset\mathscr{M}(\mathbb R^{n})$ and $f\in\mathscr{M}(\mathbb R^{n})$, then $0\leq f_{m}\uparrow f$ almost everywhere as $m\to\infty$ implies that $\|f_{m}\|_{X}\uparrow\|f\|_{X}$ as $m\to
    \infty$;
    \item $B\in\mathbb B(\mathbb R^{n})$ implies that $\chi_{B}\in X$, where $\mathbb B(\mathbb R^{n})$ is the same as (\ref{eq2.1}).
    \end{enumerate}
    Moreover, a ball quasi-Banach function space $X$ is called a \emph{ball\ Banach\ function\ space} if it satisfies
    \begin{enumerate}
    \item[(v)] for any $f,g\in X$
    \[
    \|f+g\|_{X}\leq\|f\|_{X}+\|g\|_{X};
    \]
    \item[(vi)] for any ball $B\in\mathbb B(\mathbb R^{n})$, there exists a positive constant $C_{(B)}$, depending on $B$, such that, for any $f\in X$,
    \[
    \int_{B}\vert f(x)\vert dx\leq C_{(B)}\| f\|_{X}.
    \]
    \end{enumerate}
\end{definition}
The associate space $X^{\prime}$ of any given ball Banach function space $X$ is defined as follows. Details are referred to \cite[Chapter 1, Section 2]{bennett1988interpolation} or \cite[p.9]{sawano2017hardy}.
\begin{definition}\label{def2.2}
    For any given ball Banach function space $X$, its \emph{associate space} (also called the \emph{K\"othe dual space}) $X^{\prime}$ is defined by setting
    \[
    X^{\prime}:=\{f\in\mathscr{M}(\mathbb R^{n})\colon\|f\|_{X^{\prime}}<\infty\},
    \]
    where, for any $f\in X^{\prime}$,
    \[
    \|f\|_{X^{\prime}}:=\sup\left\{ \|fg\|_{L^{1}}\colon g\in X,\ \|g\|_{X}=1 \right\},
    \]
    and $\|\cdot\|_{X^{\prime}}$ is called the \emph{associate norm} of $\|\cdot\|_{X}$.
\end{definition}
We also recall the concepts of both the convexity and the concavity of ball quasi-Banach function spaces \cite[Definition 2.6]{sawano2017hardy}.
\begin{definition}\label{def2.3}
    Let $X$ be a ball quasi-Banach function space and $p\in(0,\infty)$.
    \begin{enumerate}[(i)]
        \item The \emph{p\text{-}convexification} $X^{p}$ of $X$ is defined by setting
        \[
        X^{p}:=\{f\in\mathscr{M}(\mathbb R^{n})\colon\vert f\vert^{p}\in X\}
        \]
        equipped with the \emph{quasi-norm} $\|f\|_{X^{p}}:=\|\vert f\vert^{p}\|_{X}^{1/p}$ for any $f\in X^{p}$.
        \item The space $X$ is said to be \emph{p-convex} if there exists a positive constant $C$ such that, for any $\{f_{k}\}_{k\in\mathbb N}\subset X^{1/p}$,
        \[
        \left\| \sum_{k=1}^{\infty}\vert f_{k}\vert \right\|_{X^{1/p}}\leq C\sum_{k=1}^{\infty}\|f_{k}\|_{X^{1/p}}.
        \]
        In particular, when $C=1$, $X$ is said to be \emph{p-strictly convex}.
        \item The space $X$ is said to be \emph{p-concave} if there exists a positive constant $C$ such that, for any $\{f_{k}\}_{k\in\mathbb N}\subset X^{1/p}$,
        \[
        \sum_{k=1}^{\infty}\|f_{k}\|_{X^{1/p}}\leq C\left\| \sum_{k=1}^{\infty}\vert f_{k}\vert \right\|_{X^{1/p}}.
        \]
        In particular, when $C=1$, $X$ is said to be \emph{p-strictly concave}.
    \end{enumerate}
\end{definition}
Now, we present the definition of the local Hardy space $h_{X}(\mathbb R^{n})$ associated with $X$, which was first introduced in \cite[Definition 5.1]{sawano2017hardy}. For $N\in\mathbb Z_{+}$, $\beta\in\mathbb Z^n_{+}$, $\vert\beta\vert\leq N+1$ and $R\in(0,\infty)$, let
\[
\mathcal{S}_{N,R}(\mathbb R^{n})=\left\{  \psi\in\mathcal{S}(\mathbb R^{n})\colon{\rm supp}(\psi)\subset B(0,R),\ \int\psi\neq0,\ \|D^{\beta}\psi\|_{\infty}\leq 1 \right\}.
\]
\begin{definition}
    For any $f\in\mathcal{S}^{\prime}(\mathbb R^{n})$, the \emph{local vertical grand maximal function} $\mathcal{G}_{N,R}(f)$ of $f$ is defined by setting, for all $x\in\mathbb R^{n}$,
    \[
    \mathcal{G}_{N,R}(f)(x)=\sup_{t\in(0,1)}\left\{ \vert\psi_{t}\ast f(x)\vert\colon\psi\in\mathcal{S}_{N,R}(\mathbb R^{n}) \right\},
    \]
    and the \emph{local non-tangential grand maximal function} $\tilde{{\mathcal{G}}}_{N,R}(f)$ of $f$ is defined by setting, for all $x\in\mathbb R^{n}$,
    \[
    \tilde{{\mathcal{G}}}_{N,R}(f)(x)=\sup_{\vert x-z \vert<t<1}\left\{\vert \psi_{t}\ast f(z) \vert\colon\psi\in\mathcal{S}_{N,R}(\mathbb R^{n})\right\}.
    \]
    For convenience, we write $\mathcal{G}_{N,1}(f)=\mathcal{G}_{N}^{0}(f)$ and $\tilde{\mathcal{G}}_{N,1}(f)=\tilde{\mathcal{G}}_{N}^{0}(f)$ and also write $\mathcal{G}_{N,2^{3(10+n)}}(f)=\mathcal{G}_{N}(f)$ and $\tilde{\mathcal{G}}_{N,2^{3(10+n)}}(f)=\tilde{\mathcal{G}}_{N}(f)$.
\end{definition}
\begin{definition}\label{def2.4}
    Let $X$ be a ball quasi-Banach function space and $N\in\mathbb N$ sufficiently large. Then the \emph{local Hardy space} $h_{X}(\mathbb R^{n})$ is defined to be the set of all the $f\in\mathcal{S^{\prime}}(\mathbb R^{n})$ such that
    \[
    \|f\|_{h_{X}}:=\|\mathcal{G}_{N}(f)\|_{X}<\infty.
    \]
\end{definition}
Then we give the definition of the local-$(X,q,d)$-atom.
\begin{definition}\label{def2.6}
    Let $X$ be a ball quasi-Banach function space and $q\in(1,\infty]$. Assume that $d\in\mathbb Z_{+}$. Then a measurable function $a$ is called a \emph{local-$(X,q,d)$-atom} if 
    \begin{enumerate}[(i)]
        \item there exists a cube $Q\subset\mathbb R^{n}$ such that ${\rm supp}a\subset Q$;
        \item $\|a\|_{L^{q}}\leq\frac{\vert Q \vert^{1/q}}{\|\chi_{Q}\|_{X}}$;
        \item if $\vert Q \vert<1$, then $\int_{\mathbb R^{n}}a(x)x^{\alpha}dx=0$ for any multi-index $\alpha\in\mathbb Z^{n}_{+}$ with $\vert\alpha\vert\leq d$.
    \end{enumerate}
\end{definition}
Denote by $L^{1}_{loc}(\mathbb R^{n})$ the set of all locally integral functions on $\mathbb R^{n}$. Recall that the \emph{Hardy--Littlewood maximal operator} $\mathcal{M}$ is defined by setting, for any measurable function $f$ and any $x\in\mathbb R^{n}$,
\[
\mathcal{M}(f)(x)=\sup\limits_{x\in Q}\frac{1}{\vert Q\vert}\int_{Q}f(u)du.
\]
For any $\theta\in(0,\infty)$, the \emph{powered Hardy--Littlewood maximal operator} $\mathcal{M}^{(\theta)}$ is defined by setting, for any $f\in L^{1}_{loc}(\mathbb R^{n})$ and $x\in\mathbb R^{n}$,
\[
\mathcal{M}^{(\theta)}(f)(x):=\{\mathcal{M}(\vert f \vert^{\theta})(x)\}^{\frac{1}{\theta}}.
\]\par
Moreover, we also need a basic assumption on $X$ as follows. 
\begin{assumption}\label{ass2.7}
Let $X$ be a ball quasi-Banach function space. For some $\theta,\ s\in(0,1]$ and $\theta<s$, there exists a positive constant $C$ such that, for any $\{f_{j}\}_{j=1}^{\infty}\subset L^{1}_{loc}(\mathbb R^{n})$,
\[
\left\|  \left\{  \sum_{j=1}^{\infty}\left[ \mathcal{M}^{(\theta)}(f_{j}) \right]^{s} \right\}^{\frac{1}{s}}  \right\|_{X}\leq C\left\|  \left\{  \sum_{j=1}^{\infty}\vert f_{j} \vert^{s} \right\}^{\frac{1}{s}} \right\|_{X}.
\]
\end{assumption}
The following lemma will be used in the proof of Theorem~\ref{th4.1}.
\begin{lemma}\label{le2.8}
    Let $X$ be a ball quasi-Banach function space. Fix $q_{0}>1$. Suppose that $0<p_{0}<q_{0}$. Let $\tau=\frac{q_{0}}{p_{0}}$. If $X^{1/p_{0}}$ is a ball Banach function space and the Hardy--Littlewood maximal operator $\mathcal{M}$ is bounded on $[(X^{1/p_{0}})^{\prime}]^{1/\tau^{\prime}}$, then for all sequences of cubes $\{Q_{j}\}_{j=1}^{\infty}$ and non-negative functions $\{g_{j}\}_{j=1}^{\infty}$,
    \[
    \left\|  \sum_{j=1}^{\infty}\chi_{Q_{j}}g_{j} \right\|_{X}\leq C\left\| \sum_{j=1}^{\infty}\left( \frac{1}{\vert Q_{j}\vert}\int_{Q_{j}}g_{j}^{q_{0}}(y)dy \right)^{\frac{1}{q_{0}}}\chi_{Q_{j}} \right\|_{X}.
    \]
\end{lemma}
\begin{remark}\label{re2.9}
    In fact, by the definition of $X^{1/p_{0}}$ and $(X^{1/p_{0}})^{\prime}$, the hypothesis that the Hardy--Littlewood maximal operator $\mathcal{M}$ is bounded on $[(X^{1/p_{0}})^{\prime}]^{1/\tau^{\prime}}$ is equivalent to the assumption that for any $f\in(X^{1/p_{0}})^{\prime}$,
    \[
    \left\| \mathcal{M}^{((q_{0}/p_{0})^{\prime})}(f) \right\|_{(X^{1/p_{0}})^{\prime}}\leq C\|f\|_{(X^{1/p_{0}})^{\prime}}.
    \]
\end{remark}
In order to prove Lemma~\ref{le2.8}, we first recall some information about weights. For more details, see \cite{cruz2011weights,2001fourier}. Suppose that a weight $\omega$ is a non-negative, locally integrable function such that $0<\omega(x)<\infty$ for almost every $x\in\mathbb R^{n}$. It is said that $\omega$ is in the \emph{Muckenhoupt class} $A_{p}$ for $1<p<\infty$ if
\[
[\omega]_{A_{p}}=\sup\limits_{Q}\left(\frac{1}{Q}\int_{Q}\omega(x)dx\right)\left(\frac{1}{Q}\int_{Q}\omega(x)^{-\frac{1}{p-1}}dx\right)^{p-1}<\infty,
\]
where $Q$ is any cube in $\mathbb {R}^{n}$ and when $p=1$, a weight $\omega\in A_{1}$ if for almost everywhere $x\in\mathbb {R}^{n}$,
\[
\mathcal{M}\omega(x)\leq C\omega(x),
\]
Therefore, define the set
\[
A_{\infty}=\bigcup\limits_{1\leq p<\infty}A_{p}.
\]
Given a weight $\omega\in A_{\infty}$, define $$q_{\omega}=\inf\{q\geq1\colon \omega\in A_{q}\}.$$
A weight $\omega\in A_{\infty}$ if and only if $\omega\in RH_{r}$ for some $r>1$: that is, for every cube $Q$,
\[
\left(\frac{1}{\vert Q\vert}\int_{Q}\omega(x)^{r}dx\right)^{\frac{1}{r}}\leq \frac{C}{\vert Q\vert}\int_{Q}\omega(x)dx.
\]
Furthermore, we can obtain the property that $\omega\in RH_{r}$ if and only if $\omega^{r}\in A_{\infty}$.\par
Then, we define a family of extrapolation pairs to be a family $\mathcal{F}$ of pairs of non-negative, measurable functions $(f,g)$. Whenever we write an inequality of the form
\[
\|f\|_{X}\lesssim\|g\|_{Y},\quad (f,g)\in\mathcal{F},
\]
where $\|\cdot\|_{X}$ and $\|\cdot\|_{Y}$ are quasi-norms in ball quasi-Banach function spaces, we mean that this equality holds for every pair $(f,g)$ in $\mathcal{F}$ such that $\|f\|_{X}<\infty$ and the constant is independent of the pair $(f,g)$. We now prove the extrapolation theorem into the scale of ball quasi-Banach function spaces. For more details about extrapolation and the theory of Rubio de Francia, see \cite{cruz2011weights}. In fact, from the following lemma and \cite[Lemma 4.7]{cruz2020}, we can immediately obtain Lemma~\ref{le2.8}.
\begin{lemma}\label{le2.10}
    Let $X$ be a ball quasi-Banach function space. Given $0<p<q$. Let $\tau=\frac{q}{p}$. Suppose that for all $\omega\in RH_{\tau^{\prime}}$,
    \[
    \|f\|_{L^{p}_{\omega}}\lesssim\|g\|_{L^{p}_{\omega}}, \quad (f,g)\in\mathcal{F}.
    \]
    If $X^{1/p}$ is a ball Banach function space and the Hardy--Littlewood maximal operator $\mathcal{M}$ is bounded on $[(X^{1/p})^{\prime}]^{1/\tau^{\prime}}$, then for any $f\in X$ and $(f,g)\in\mathcal{F}$,
    \[
    \|f\|_{X}\lesssim\|g\|_{X}.
    \]
\end{lemma}
\begin{proof}
    For any non-negative $h$, we define the Rubio de Francia iteration algorithm by
    \[
    \mathcal{R}h(x):=\sum_{k\in\mathbb Z_{+}}\frac{\mathcal{M}^{k}(h)(x)}{2^{k}\|\mathcal{M}\|^{k}_{[(X^{\frac{1}{p}})^{\prime}]^{\frac{1}{\tau^{\prime}}}}},
    \]
    where $\mathcal{M}^{0}(h):=\vert h\vert$ and, for any $k\in\mathbb N$, $\mathcal{M}^{k}:=\mathcal{M}\circ\dots\circ\mathcal{M}$ is $k$ iterations of $\mathcal{M}$. Then, $\mathcal{R}h$ has the following properties:
    \begin{enumerate}[(i)]
    \item $h(x)\leq\mathcal{R}h(x)$;
    \item $\|\mathcal{R}h\|_{[(X^{\frac{1}{p}})^{\prime}]^{\frac{1}{\tau^{\prime}}}}\leq 2\|h\|_{[(X^{\frac{1}{p}})^{\prime}]^{\frac{1}{\tau^{\prime}}}}$;
    \item $\mathcal{M}(\mathcal{R}h)\leq 2\|\mathcal{M}\|_{([(X^{\frac{1}{q}})^{\prime}]^{\frac{1}{\tau^{\prime}}}}\mathcal{R}h$, namely, $\mathcal{R}h\in A_{1}$ and
    \[
    [\mathcal{R}h]_{A_{1}}\leq 2\|\mathcal{M}\|_{[(X^{\frac{1}{p}})^{\prime}]^{\frac{1}{\tau^{\prime}}}};
    \]
    \item $\mathcal{R}(h^{\tau^{\prime}})\in A_{1}$ and $(\mathcal{R}(h^{\tau^{\prime}}))^{\frac{1}{\tau^{\prime}}}\in RH_{\tau^{\prime}}$.
    \end{enumerate}\par
    Now fix $(f,g)\in\mathcal{F}$ such that $\|f\|_{X}<\infty$. Then by duality, we have
    \[
    \begin{aligned}
        \|f\|_{X}^{p}
        &=\left\|  \vert f\vert^{p}  \right\|_{X^{\frac{1}{p}}}\\
        &\leq\sup_{\|h\|_{(X^{1/p})^{\prime}}=1}\int_{\mathbb R^{n}}\vert f(x)\vert^{p}\vert h(x)\vert dx\\
        &\leq\sup_{\|h\|_{(X^{1/p})^{\prime}}=1}\int_{\mathbb R^{n}}\vert f(x)\vert^{p}\left[\mathcal{R}\left(\vert h\vert^{\tau^{\prime}}\right)\right]^{\frac{1}{\tau^{\prime}}} dx.
    \end{aligned} 
    \]
    Let $H=\left[\mathcal{R}\left(\vert h\vert^{\tau^{\prime}}\right)\right]^{\frac{1}{\tau^{\prime}}}$. By \cite[Lemma 3.22]{yiqun2023boundedness} and the above properties,
    \[
    \begin{aligned}
        \sup_{\|h\|_{(X^{1/p})^{\prime}}=1}\int_{\mathbb R^{n}}\vert f(x)\vert^{p}H(x)dx
        &\leq\sup_{\|h\|_{(X^{1/p})^{\prime}}=1}\left\|  \vert f\vert^{p}  \right\|_{X^{1/p}}\left\| H\right\|_{(X^{1/p})^{\prime}}\\
        &=\sup_{\|h\|_{(X^{1/p})^{\prime}}=1}\|f\|_{X}^{p}\left\| \mathcal{R}\left(\vert h\vert^{\tau^{\prime}}\right)  \right\|_{[(X^{1/p})^{\prime}]^{1/\tau^{\prime}}}^{\frac{1}{\tau^{\prime}}}\\
        &\leq\sup_{\|h\|_{(X^{1/p})^{\prime}}=1} 2\|f\|_{X}^{p}\left\|\vert h\vert^{\tau^{\prime}}\right\|_{[(X^{1/p})^{\prime}]^{1/\tau^{\prime}}}^{\frac{1}{\tau^{\prime}}}\\
        &=\sup_{\|h\|_{(X^{1/p})^{\prime}}=1}2\|f\|_{X}^{p}\|h\|_{(X^{1/p})^{\prime}}\\
        &<\infty.
    \end{aligned}
    \]
    Therefore, by applying the fact that $H\in RH_{\tau^{\prime}}$ and the hypothesis of the lemma, we obtain that
    \[
    \begin{aligned}
        \|f\|_{X}^{p}
        &\leq\sup_{\|h\|_{(X^{1/p})^{\prime}}=1}\int_{\mathbb R^{n}}\vert f(x)\vert^{p}H(x) dx\\
        &\lesssim\sup_{\|h\|_{(X^{1/p})^{\prime}}=1}\int_{\mathbb R^{n}}\vert g(x)\vert^{p}H(x)dx\\
        &\lesssim\sup_{\|h\|_{(X^{1/p})^{\prime}}=1}\|g\|_{X}^{p}\|h\|_{(X^{1/p})^{\prime}}\\
        &\lesssim\|g\|_{X}^{p},
    \end{aligned}
    \]
    where the penultimate inequality follows from the above estimate with $g$ in place of $f$. Therefore, we complete the proof.
\end{proof}

\section{Maximal function characterizations}\label{se3}
In this section, we will give several equivalent characterizations for the local Hardy space $h_{X}(\mathbb R^{n})$ associated with $X$, which will play a key role in the proof of atomic decomposition characterizations in the next section.
For more details about real-variable characterizations of Hardy type spaces, see for instance \cite{CFJY,CFYY,LFFY}. 
\par

In order to obtain the equivalence between the local vertical grand maximal function and the local non-tangential grand maximal function, we need introduce the following functions as an intermediate step. For $\psi\in\mathcal{S}(\mathbb R^{n})$, we write $\psi_{j}(x)=2^{jn}\psi(2^{j}x)$ for all $j\in\mathbb N$. 
\begin{definition}\label{def3.1}
    Let $\psi_{0}\in\mathcal{S}(\mathbb R^{n})$ with $\int\psi_{0}(x)dx\neq0$ and $L\in[0,\infty)$. The \emph{local vertical maximal function} $\psi_{0}(f)(x)$ is defined by
    \[
    \psi_{0}^{+}(f)(x):=\sup_{j\in\mathbb N}\left\vert 
(\psi_{0})_{j}\ast f(x)  \right\vert.
    \]
    The \emph{local tangential Peetre-type maximal function} $\psi_{0,L}^{**}(f)(x)$ is defined by
    \[
    \psi_{0,L}^{**}(f)(x):=\sup_{j\in\mathbb N,y\in\mathbb R^{n}}\frac{\left\vert (\psi_{0})_{j}\ast f(x-y) \right\vert}{(1+2^{j}\vert y\vert)^{L}}.
    \]
    The \emph{local non-tangential maximal function} $(\psi_{0})_{\nabla}^{*}(f)(x)$ of $f$ is defined by
    \[
    (\psi_{0})_{\nabla}^{*}(f)(x)=\sup_{\vert x-y\vert<t<1}\vert(\psi_{0})_{t}\ast f(y)\vert.
    \]
\end{definition}
\begin{remark}
    Obviously, for any $x\in\mathbb R^{n}$, we have
    \[
    \psi_{0}^{+}(f)(x)\leq(\psi_{0})_{\nabla}^{*}(f)(x)\lesssim\psi_{0,L}^{**}(f)(x).
    \]
\end{remark}
Next we present the main theorem in this section.
\begin{theorem}\label{th3.3}
    Let $X$ be a ball-quasi Banach function space satisfying Assumption~\ref{ass2.7} and $\psi_{0}\in\mathcal{S}(\mathbb R^{n})$ such that $\int_{\mathbb R^{n}}\psi_{0}(x)dx\neq0$. Then for any large enough integer $N$ and all $f\in\mathcal{S}^{\prime}(\mathbb R^{n})$, 
    \[
    \begin{aligned}
        \left\| f \right\|_{h_{X}}
        &\sim\left\|\psi_{0}^{+}(f)\right\|_{X}\sim\left\|\psi_{0,L}^{**}(f)\right\|_{X}\sim\left\| (\psi_{0})_{\nabla}^{*}(f) \right\|_{X}\\
        &\sim\left\|\tilde{\mathcal{G}}_{N}(f)\right\|_{X}\sim\left\|\mathcal{G}_{N}^{0}(f)\right\|_{X}\sim\left\|\tilde{\mathcal{G}}_{N}^{0}(f)\right\|_{X}
    \end{aligned}
    \]
    where the implicit equivalent positive constants are independent of $f$.
\end{theorem}
In order to obtain equivalent characterizations as above, we give the next lemma, which are inspired by \cite[Lemma 3.2]{nakai2012hardy}.
\begin{lemma}\label{le3.4}
    For all $f\in\mathcal{S}^{\prime}(\mathbb R^{n})$ and $0<\theta<1$, there exists $L_{\theta}$ so that for all $L\geq L_{\theta}$, we have
    \[
    \psi_{0,L}^{**}(f)(x)\lesssim\mathcal{M}^{(\theta)}(\psi_{0}^{+}(f))(x).
    \]
\end{lemma}
\begin{proof}
    Denote by $\delta\in\mathcal{S}^{\prime}(\mathbb R^{n})$ the Dirac delta and let $L\gg1$. Let $\varphi, \Psi\in\mathcal{S}(\mathbb R^{n})$ be such that
    \[
    \varphi\ast\psi_{0}+\sum_{k=1}^{\infty}\Psi_{k}\ast((\psi_{0})_{k}-(\psi_{0})_{k-1})=\delta,\quad
    \int_{\mathbb R^{n}}x^{\alpha}\Psi(x)dx=0\ (\vert\alpha\vert\leq n+3L+2)
    \]
    in the topology of $\mathcal{S}^{\prime}(\mathbb R^{n})$. We refer to \cite[Theorem 1.6]{rychkov2001littlewood} for more details on the pair $(\varphi,\Psi)$. Then, we have
    \[
    \varphi_{j}\ast(\psi_{0})_{j}+\sum_{k=j+1}^{\infty}\Psi_{k}\ast((\psi_{0})_{k}-(\psi_{0})_{k-1})=\delta
    \]
    by the dilation. Then, by the triangle inequality, we can obtain
    \[
    \begin{aligned}
    \frac{(\psi_{0})_{j}\ast f(x-y)}{(1+2^{j}\vert y\vert)^{L}}
    &\leq\frac{\vert (\psi_{0})_{j}\ast\varphi_{j}\ast(\psi_{0})_{j}\ast f(x-y)\vert}{(1+2^{j}\vert y\vert)^{L}}\\
    &+\sum_{k=j+1}^{\infty}\frac{\vert (\psi_{0})_{j}\ast\Psi_{k}\ast((\psi_{0})_{k}-(\psi_{0})_{k-1})\ast f(x-y)\vert}{(1+2^{j}\vert y\vert)^{L}}\\
    &=\uppercase\expandafter{\romannumeral1}+\uppercase\expandafter{\romannumeral2}.
    \end{aligned}
    \]
    For $\uppercase\expandafter{\romannumeral2}$, we have
    \[
    \begin{aligned}
    &\frac{\vert (\psi_{0})_{j}\ast\Psi_{k}\ast((\psi_{0})_{k}-(\psi_{0})_{k-1})\ast f(x-y)\vert}{(1+2^{j}\vert y\vert)^{L}}\\
    &\leq\int_{\mathbb R^{n}}\frac{\vert (\psi_{0})_{j}\ast\Psi_{k}(z)((\psi_{0})_{k}-(\psi_{0})_{k-1})\ast f(x-y-z)\vert}{(1+2^{j}\vert y\vert)^{L}}dz.
    \end{aligned}
    \]
    From \cite[Lemma 2.5]{nakai2012hardy}, we obtain
    \[
    \vert(\psi_{0})_{j}\ast\Psi_{k}(z)\vert\lesssim2^{jn-3(k-j)L}(1+2^{j}\vert z\vert)^{-L}.
    \]
    Thus, it follows that
    \[
    \begin{aligned}
        &\frac{\vert (\psi_{0})_{j}\ast\Psi_{k}\ast((\psi_{0})_{k}-(\psi_{0})_{k-1})\ast f(x-y)\vert}{(1+2^{j}\vert y\vert)^{L}}\\
        &\lesssim\int_{\mathbb R^{n}}\frac{2^{jn-3(k-j)L}\vert((\psi_{0})_{k}-(\psi_{0})_{k-1})\ast f(x-y-z)\vert}{(1+2^{j}\vert y\vert)^{L}(1+2^{j}\vert z\vert)^{L}}dz.
    \end{aligned}
    \]
    By applying the fact that $(1+\vert x\vert)(1+\vert y\vert)\geq1+\vert x-y\vert$, we obtain that
    \[
    \begin{aligned}
    &\frac{\vert (\psi_{0})_{j}\ast\Psi_{k}\ast((\psi_{0})_{k}-(\psi_{0})_{k-1})\ast f(x-y)\vert}{(1+2^{j}\vert y\vert)^{L}}\\
    &\lesssim\int_{\mathbb R^{n}}\frac{2^{jn-3(k-j)L}\vert((\psi_{0})_{k}-(\psi_{0})_{k-1})\ast f(x-y-z)\vert}{(1+2^{j}\vert y+z\vert)^{L}}dz\\
    &\lesssim(\psi_{0,L}^{**}(f)(x))^{1-\theta}\int_{\mathbb R^{n}}\frac{2^{jn-3(k-j)L\theta}}{(1+2^{j}\vert y+z\vert)^{L\theta}}\\
    &\times\left(\left\vert(\psi_{0})_{k}\ast f(x-y-z)\right\vert^{\theta}+\left\vert(\psi_{0})_{k-1}\ast f(x-y-z)\right\vert^{\theta}\right)dz.
    \end{aligned}
    \]
    For $\uppercase\expandafter{\romannumeral1}$, we can repeat the similar argument as used above.
    Therefore, we have
    \[
    \begin{aligned}
        &\frac{(\psi_{0})_{j}\ast f(x-y)}{(1+2^{j}\vert y\vert)^{L}}\\
        &\lesssim(\psi_{0,L}^{**}(f)(x))^{1-\theta}\sum_{k=j}^{\infty}\int_{\mathbb R^{n}}\frac{2^{jn-3(k-j)L\theta}}{(1+2^{j}\vert y+z\vert)^{L\theta}}\left(\left\vert(\psi_{0})_{k}\ast f(x-y-z)\right\vert^{\theta}\right)dz.
    \end{aligned}
    \]
    Let 
    \[
    G_{L,\theta}(f)(x)\equiv\sup_{y\in\mathbb R^{n}}\sup_{j\in\mathbb N}\sum_{k=j}^{\infty}\int_{\mathbb R^{n}}\frac{2^{jn-3(k-j)L\theta}}{(1+2^{j}\vert y+z\vert)^{L\theta}}\left(\left\vert(\psi_{0})_{k}\ast f(x-y-z)\right\vert^{\theta}\right)dz.
    \]
    Thus, we obtain
    \[
    \psi_{0,L}^{**}(f)(x)\lesssim(\psi_{0,L}^{**}(f)(x))^{1-\theta}G_{L,\theta}(f)(x).
    \]
    Then, we claim that
    \begin{equation}\label{eq1}
    (\psi_{0,L}^{**}(f)(x))^{\theta}\lesssim G_{L,\theta}(f)(x).
    \end{equation}
    It follows that
    \begin{equation}\label{eq2}
        (\psi_{0,L}^{**}(f)(x))^{\theta}\lesssim\mathcal{M}\left[\sup_{k\in\mathbb N}\left\vert (\psi_{0})_{k}\ast f  \right\vert^{\theta} \right](x).
    \end{equation}
    The proof of (\ref{eq2}) is based on the technique of \cite[Chapter 5]{stromberg1989}.
    Now, we try to prove the claim. If $G_{L,\theta}(f)(x)=\infty$, we trivially have (\ref{eq1}). If $G_{L,\theta}(f)(x)<\infty$, it suffices to show $\psi_{0,L}^{**}(f)(x)<\infty$. Since $f\in\mathcal{S}^{\prime}(\mathbb R^{n})$, there exists $L_{f}\in\mathbb N$ such that $\psi_{0,L_{f}}^{**}(f)(x)<\infty$. If $L\geq L_{f}$, $\psi_{0,L}^{**}(f)(x)\leq\psi_{0,L_{f}}^{**}(f)(x)<\infty$. If $L\leq L_{f}$, notice that $G_{L_{f},\theta}(f)(x)\leq G_{L,\theta}(f)(x)<\infty$. Thus, we have
    \[
    \left\vert (\psi_{0})_{j}\ast f(x) \right\vert^{\theta}\leq(\psi_{0,L_{f}}^{**}(f)(x))^{\theta}\lesssim G_{L_{f},\theta}(f)(x).
    \]
    Then, we conclude that
    \[
    (\psi_{0,L}^{**}(f)(x))^{\theta}\lesssim G_{L,\theta}(f)(x)<\infty.
    \]
    Therefore, we finish the proof of Lemma~\ref{le3.4}.
\end{proof}
Finally, by Lemma~\ref{le3.4} and Assumption~\ref{ass2.7}, we can obtain Theorem~\ref{th3.3}. Since the proof of Theorem~\ref{th3.3} is similar to \cite[Theorem 3.14]{yang2011weighted}, we omit the details here.

\section{Atomic characterization}\label{se4}
In this section, we will establish the atomic decomposition characterization of $h_{X}(\mathbb R^{n})$ associated with $X$. Firstly, we can present the reconstruction theorem and the atom decomposition theorem, respectively. Recall that $X$ is said to have an \emph{absolutely continuous quasi-norm} if, for any $f\in X$ and any measurable subsets $\{E_{j}\}_{j\in\mathbb N}\subset\mathbb R^{n}$ with both $E_{j+1}\subset E_{j}$ for any $j\in\mathbb N$ and $\bigcap_{j\in\mathbb N}E_{j}=\emptyset$,  $ \|f\chi_{E_{j}} \|_{X}\downarrow0$ as $j\to\infty$.
\begin{theorem}\label{th4.1}
        Let $X$ be a ball quasi-Banach function space satisfying Assumption~{\ref{ass2.7}}. Let $p_{0},\ q_{0}$ and $\tau$ be as in Lemma~\ref{le2.8}. Assume that $X^{1/p_{0}}$ is a ball Banach function space and the Hardy--Littlewood maximal operator $\mathcal{M}$ is bounded on $[(X^{1/p_{0}})^{\prime}]^{1/\tau^{\prime}}$.
        Further assume that $X$ has an absolutely continuous quasi-norm. Suppose that $1<q\leq\infty$ and $d\in\mathbb Z_{+}$. Given countable collections of cubes $\{Q_{j}\}_{j=1}^{\infty}$, of non-negative coefficients $\{\lambda_{j}\}_{j=1}^{\infty}$ and of the local-$(X,q,d)$-atoms $\{a_{j}\}_{j=1}^{\infty}$, if 
    \[
    \left\|  \sum_{j=1}^{\infty}\frac{\lambda_{j}\chi_{Q_{j}}}{\|\chi_{Q_{j}}\|_{X}} \right\|_{X}<\infty,
    \]
    then the series $f=\sum_{j=1}^{\infty}\lambda_{j}a_{j}$ converges in $h_{X}(\mathbb R^{n})$ and satisfies
    \[
    \|f\|_{h_{X}}\leq C\left\| \sum_{j=1}^{\infty}\frac{\lambda_{j}\chi_{Q_{j}}}{\|\chi_{Q_{j}}\|_{X}} \right\|_{X}.
    \]
\end{theorem}
\begin{theorem}\label{th4.2}
    Let $X$ be a ball quasi-Banach function space satisfying Assumption~{\ref{ass2.7}}. Assume that $X$ has an absolutely continuous quasi-norm. Suppose that $1<q\leq\infty$, $\eta\in(0,\infty)$ and $d\in\mathbb Z_{+}$. If $f\in h_{X}(\mathbb R^{n})$, then there exist non-negative coefficients $\{\lambda_{j}\}_{j=1}^{\infty}$ and the local-$(X,q,d)$-atoms $\{a_{j}\}_{j=1}^{\infty}$ such that $f=\sum_{j=1}^{\infty}\lambda_{j}a_{j}$, where the series converges almost everywhere and in $\mathcal{S}^{\prime}(\mathbb R^{n})$, and that
    \[
    \left\|  \left( \sum_{j=1}^{\infty}\left( \frac{\lambda_{j}\chi_{Q_{j}}}{\|\chi_{Q_{j}}\|_{X}} \right)^{\eta} \right)^{\frac{1}{\eta}} \right\|_{X}\leq C_{\eta}\|f\|_{h_{X}}.
    \]
\end{theorem}
Hence, we can immediately obtain the corollary as follows.
\begin{corollary}\label{cor4.3}
    Let $X$ be a ball quasi-Banach function space satisfying Assumption~{\ref{ass2.7}}. Let $p_{0},\ q_{0}$ and $\tau$ be as in Lemma~\ref{le2.8}. Assume that $X^{1/p_{0}}$ is a ball Banach function space and the Hardy--Littlewood maximal operator $\mathcal{M}$ is bounded on $[(X^{1/p_{0}})^{\prime}]^{1/\tau^{\prime}}$. Further assume that $X$ has an absolutely continuous quasi-norm. Suppose that $1<q\leq\infty$ and $d\in\mathbb Z_{+}$. Then $f\in\mathcal{S}^{\prime}(\mathbb R^{n})$ is in $h_{X}(\mathbb R^{n})$ if and only if there exists non-negative coefficients $\{\lambda_{j}\}_{j=1}^{\infty}$ and the local-$(X,q,d)$-atoms $\{a_{j}\}_{j=1}^{\infty}$ such that $f=\sum_{j=1}^{\infty}\lambda_{j}a_{j}$, where the series converges in $h_{X}(\mathbb R^{n})$, and that
    \[
    \|f\|_{h_{X}}\sim\inf\left\{  \left\|\sum_{j=1}^{\infty}\frac{\lambda_{j}\chi_{Q_{j}}}{\|\chi_{Q_{j}}\|_{X}}\right\|_{X}\colon f=\sum_{j=1}^{\infty}\lambda_{j}a_{j} \right\},
    \]
    where the infimum is taken over all decomposition of $f$ as above.
\end{corollary}
\begin{proof}[{\bf Proof of Theorem~{\upshape\ref{th4.1}}}]

    Fix $\psi\in\mathcal{S}_{N,1}$ with $\int\psi(x)dx\neq0$. By Theorem~\ref{th3.3}, we have that
    \[
    \|f\|_{h_{X}}\sim\|\mathcal{G}_{N}^{0}(f)\|_{X}\equiv\left\|\sup_{0<t<1}\vert\psi_{t}\ast f\vert\right\|_{X}.
    \]
    Let every $\{a_{j}\}_{j=1}^{\infty}$ be a local-$(X,q,d)$-atom which is support in cubes $\{Q_{j}\}_{j=1}^{\infty}$. First, suppose that $\{\lambda_{j}\}_{j=1}^{\infty}$ has only a finite number of non-zero entries. Then, we consider two cases.\par
    When $\vert Q_{j}\vert<1$, we claim that
    \[
    \mathcal{G}_{N}^{0}\left[ \sum_{j=1}^{\infty}\lambda_{j}a_{j} \right](x)\leq C\sum_{j=1}^{\infty}\lambda_{j}\mathcal{M}(a_{j})(x)\chi_{Q_{j}^{*}}(x)+\frac{\lambda_{j}}{\|\chi_{Q_{j}}\|_{X}}\mathcal{M}(\chi_{Q_{j}})(x)^{\frac{n+d+1}{n}}\chi_{(Q_{j}^{*})^{c}}.
    \]
    In fact, when $x\in Q_{j}^{*}$, for any local-$(X,q,d)$-atom $a_{j}$, we just need the pointwise estimate
    \[
    \mathcal{G}_{N}^{0}(a_{j})(x)\leq C\mathcal{M}(a_{j}).
    \]
    Then, it suffices to prove that for all $x\in(Q_{j}^{*})^{c}$,
    \[
    \mathcal{G}_{N}^{0}(a_{j})(x)\leq C\frac{l(Q_{j})^{n+d+1}}{\|\chi_{Q_{j}}\|_{X}\vert x-x_{j}\vert^{n+d+1}}.
    \]
    For any $0<t<1$, let $P$ be the Taylor expansion of $\psi$ at the point $\frac{x-x_{j}}{t}$ with degree $d$. By applying the Taylor remainder theorem, we have
    \[
    \begin{aligned}
        &\left\vert \psi\left(\frac{x-y}{t}\right)-P\left(\frac{x-x_{j}}{t}\right) \right\vert\\
        &\leq C\sum_{\vert\alpha\vert=d+1}\left\vert D^{\alpha}\psi\left( \frac{\theta(x-y)+(1-\theta)(x-x_{j})}{t} \right) \right\vert\left\vert\frac{x_{j}-y}{t}\right\vert^{d+1},
        \end{aligned}
    \]
    where multi-index $\alpha\in\mathbb Z^{n}_{+}$ and $\theta\in(0,1)$. Since $0<t<1$ and $x\in(Q_{j}^{*})^{c}$, then we can get that ${\rm supp}(a_{j}\ast\psi_{t} )\subset B(x_{j},2\sqrt{n})$ and that $a_{j}\ast\psi_{t}(x)\neq0$ implies that $t>\frac{\vert x-x_{j}\vert}{2}$. Thus, for any $x\in(Q_{j}^{*})^{c}$, we have that
    \[
    \begin{aligned}
        \vert (a_{j}\ast\psi_{t})(x)\vert
        &=\left\vert t^{-n}\int_{\mathbb R^{n}}\int_{\mathbb R^{n}}a_{j}(y)\left( \psi\left(\frac{x-y}{t}\right)-P\left(\frac{x-x_{j}}{t}\right) \right)dy \right\vert\\
        &\leq C\vert x-x_{j}\vert^{-(n+d+1)}\chi_{(Q_{j}^{*})^{c}}\int_{Q_{j}}\vert a_{j}(y)\vert\vert y-x_{j}\vert^{d+1}dy\\
        &\leq C\vert x-x_{j}\vert^{-(n+d+1)}l(Q_{j})^{n+d+1}\|\chi_{Q_{j}}\|_{X}^{-1}\chi_{(Q_{j}^{*})^{c}}.
    \end{aligned}
    \]
    Hence, we finish the proof of the claim. Then by applying the $L^{q}$-boundedness of the maximal operator $\mathcal{M}$, we obtain that
    \[
    \begin{aligned}
        &\left( \frac{1}{\vert Q_{j}^{*}\vert}\int_{Q_{j}^{*}}\vert \mathcal{M}(a_{j})(x)\chi_{Q_{j}^{*}} \vert^{q} \right)^{\frac{1}{q}}
        \leq\frac{1}{\vert Q_{j}^{*}\vert^{\frac{1}{q}}}\| \mathcal{M}(a_{j}) \|_{L^{q}}\\
        &\leq \frac{C}{\vert Q_{j}^{*}\vert^{\frac{1}{q}}}\| a_{j} \|_{L^{q}}
        \leq \frac{C}{\|\chi_{Q_{j}}\|_{X}}.
    \end{aligned}
    \]
    Choose $\gamma>1$. Then by the Assumption~{\ref{ass2.7}} and Lemma~{\ref{le2.8}}, we conclude that
    \[
    \begin{aligned}
        &\left\|  \mathcal{G}_{N}^{0}\left( \sum_{j=1}^{\infty}\lambda_{j}a_{j} \right)  \right\|_{X}\\
        &\leq C\left\|  \sum_{j=1}^{\infty}\lambda_{j}\left( \frac{1}{\vert Q_{j}^{*}\vert}\int_{Q_{j}^{*}}\vert \mathcal{M}(a_{j})(x)\chi_{Q_{j}^{*}} \vert^{q} \right)^{\frac{1}{q}}\chi_{Q_{j}^{*}} \right\|_{X}+C\left\| \sum_{j=1}^{\infty}\frac{\lambda_{j}\chi_{(Q_{j}^{*})^{c}}}{\|\chi_{Q_{j}}\|_{X}}\left( \mathcal{M}\chi_{Q_{j}} \right)^{\frac{n+d+1}{n}}\right\|_{X}\\
        &\leq C\left\| \sum_{j=1}^{\infty}\frac{\lambda_{j}\chi_{Q_{j}^{*}}}{\|\chi_{Q_{j}}\|_{X}} \right\|_{X}\\
        &\leq C\left\| \left(  \sum_{j=1}^{\infty}\frac{\lambda_{j}\mathcal{M}^{\gamma}(\chi_{Q_{j}})}{\|\chi_{Q_{j}}\|_{X}} \right)^{\frac{1}{\gamma}} \right\|_{X^{\gamma}}^{\gamma}\\
        &\leq C\left\| \sum_{j=1}^{\infty}\frac{\lambda_{j}\chi_{Q_{j}}}{\|\chi_{Q_{j}}\|_{X}} \right\|_{X}.
    \end{aligned}
    \]
    When $q=\infty$, it is obvious that for any $x\in Q_{j}^{*}$,
    \[
    \mathcal{M}(a_{j})(x)\leq C\|a_{j}\|_{L^{\infty}}\leq\frac{C}{\|\chi_{Q_{j}}\|_{X}}.
    \]
    Repeating the similar argument as used above, we can obtain that
    \[
    \begin{aligned}
        &\left\|\mathcal{G}_{N}^{0}\left( \sum_{j=1}^{\infty}\lambda_{j}a_{j} \right)\right\|_{X}\\
       &\leq C\left\|  \sum_{j=1}^{\infty}\frac{\lambda_{j}\chi_{Q_{j}^{*}}}{\|\chi_{Q_{j}}\|_{X}} \right\|_{X}+C\left\| \sum_{j=1}^{\infty}\frac{\lambda_{j}\chi_{(Q_{j}^{*})^{c}}}{\|\chi_{Q_{j}}\|_{X}}\left( \mathcal{M}\chi_{Q_{j}} \right)^{\frac{n+d+1}{n}}\right\|_{X}\\
        &\leq C\left\| \sum_{j=1}^{\infty}\frac{\lambda_{j}\chi_{Q_{j}^{*}}}{\|\chi_{Q_{j}}\|_{X}} \right\|_{X}\\
        &\leq C\left\| \sum_{j=1}^{\infty}\frac{\lambda_{j}\chi_{Q_{j}}}{\|\chi_{Q_{j}}\|_{X}} \right\|_{X}.
    \end{aligned}
    \]\par
    When $\vert Q_{j} \vert\geq1$, let $\Bar{Q}_{j}=Q_{j}(x_{j},l(Q_{j})+2)$. We observe that
    \[
    \mathcal{G}_{N}^{0}\left( \sum_{j=1}^{\infty}\lambda_{j}a_{j} \right)(x)\leq C\sum_{j=1}^{\infty}\lambda_{j}\mathcal{M}(a_{j})(x)\chi_{\Bar{Q}_{j}}(x).
    \]
    Thus, by applying the similar but easier argument and the fact that ${\rm supp}(a_{j}\ast\psi_{t})\subset\Bar{Q}_{j}\subset100Q_{j}$, we can obtain that
    \[
    \begin{aligned}
        \left\|\mathcal{G}_{N}^{0}\left( \sum_{j=1}^{\infty}\lambda_{j}a_{j} \right)\right\|_{X}
        &\leq C\left\| \lambda_{j}\mathcal{M}(a_{j})\chi_{\Bar{Q}_{j}} \right\|_{X}\\
        &\leq C\left\| \sum_{j=1}^{\infty}\lambda_{j}\left( \frac{1}{\vert \Bar Q_{j}\vert}\int_{\Bar Q_{j}}\vert\mathcal{M}(a_j)\chi_{\Bar Q_{j}}\vert^{q} dx\right)^{\frac{1}{q}}\chi_{\Bar{Q}_{j}} \right\|_{X}\\
        &\leq C\left\| \sum_{j=1}^{\infty}\frac{\lambda_{j}\chi_{100Q_{j}}}{\|\chi_{Q_{j}}\|_{X}}  \right\|_{X}\\
        &\leq C\left\|\sum_{j=1}^{\infty}\frac{\lambda_{j}\chi_{Q_{j}}}{\|\chi_{Q_{j}}\|_{X}} \right\|_{X}.
    \end{aligned}
    \]
    When $q=\infty$, we can obtain the desired result similarly.\par
    In the end, we extend the result to the general case. Given countable collections of cubes $\{Q_{j}\}_{j=1}^{\infty}$, of non-negative coefficients $\{\lambda_{j}\}_{j=1}^{\infty}$ and of the local-$(X,q,d)$-atom $\{a_{j}\}_{j=1}^{\infty}$, notice that 
    \[
    \left\|  \sum_{j=1}^{\infty}\frac{\lambda_{j}\chi_{Q_{j}}}{\|\chi_{Q_{j}}\|_{X}} \right\|_{X}<\infty
    \]
    and that for $1\leq m\leq n<\infty$,
    \[
    \left\|\mathcal{G}_{N}^{0}\left( \sum_{j=m}^{n}\lambda_{j}a_{j} \right)\right\|_{X}\leq C\left\|\sum_{j=m}^{n}\frac{\lambda_{j}\chi_{Q_{j}}}{\|\chi_{Q_{j}}\|_{X}} \right\|_{X}.
    \]
    Therefore, the sequence $\{\sum_{j=1}^{n}\lambda_{j}a_{j}\}_{n=1}^{\infty}$ is Cauchy and converges in an element $f\in h_{X}(\mathbb R^{n})$. Moreover, we claim that $h_{X}$-norm is stronger than the topology of $\mathcal{S}^{\prime}(\mathbb R^{n})$. In fact, for any $\varphi\in\mathcal{ S}(\mathbb R^{n})$ and ${\rm supp}\varphi\subset B_{0}=B(0,1)$, we can get that
    \[
    \begin{aligned}
        \vert \langle f,\varphi\rangle \vert^{p}
        &=\vert f\ast\tilde \varphi(0)\vert^{p}\\
        &\leq C\inf_{y\in B_{0}}\mathcal{G}_{N}^{0}(f)(y)^{p}\\
        &\leq C\frac{1}{\vert B_{0}\vert}\int_{B_{0}}\mathcal{G}_{N}^{0}(f)(y)^{p}dy\\
        &\leq C_{B_{0}}\|\mathcal{G}_{N}^{0}(f)^{p}\|_{X^{\frac{1}{p}}}\\
        &\leq C_{B_{0}}\|\mathcal{G}_{N}^{0}(f)\|_{X}^{p},
    \end{aligned}
    \]
    where $\tilde\varphi(x)=\varphi(-x)$ and we choose $p<1$ such that $X^{\frac{1}{p}}$ is a ball Banach function space. Hence, the sequence $\{\lambda_{j}a_{j}\}_{j=1}^{\infty}$ also converge to $f$ in $\mathcal{S}^{\prime}(\mathbb R^{n})$. Furthermore,
    \[
    \|f\|_{h_{X}}\leq C\lim_{n\to\infty}\left\| \mathcal{G}_{N}^{0}\left( \sum_{j=1}^{n}\lambda_{j}a_{j} \right) \right\|_{X}\leq C\left\| \sum_{j=1}^{\infty}\frac{\lambda_{j}\chi_{Q_{j}}}{\|\chi_{Q_{j}}\|_{X}} \right\|_{X}.
    \]
\end{proof}
Next we will prove the atomic decomposition. Since its proof is similar to that of \cite[Theorem 4.3]{tan2023real}, we only sketch some important steps here. For any $s\in\mathbb Z_{+}$, we use $\mathcal{P}_{s}(\mathbb R^{n})$ to denote the set of all the polynomials on $\mathbb R^{n}$ with total degree not greater than s;
\begin{proof}[{\bf Proof of Theorem~{\upshape\ref{th4.2}}}]
    First assume that $f\in h_{X}(\mathbb R^{n})\cap L^{2}(\mathbb R^{n})$. For each $k\in\mathbb Z$, set $\Omega_{k}=\{x\in\mathbb R^{n}\colon\mathcal{G}_{N}(f)(x)>2^{k}\}$. Then it follows that $\Omega_{k+1}\subset\Omega_{k}$. By \cite[Lemma 4.5]{tan2023real}, $f$ admits a Calder\'on-Zygmund decomposition of degree $d$ and height $2^{k}$ associated with $\mathcal{G}_{N}(f)$,
\[
f=g^{k}+\sum_{i}b^{k}_{i},\quad {\rm in}\ \mathcal{S}^{\prime}(\mathbb R^{n}),
\]
where $b^{k}_{i}=(f-P^{k}_{i})\eta_{i}^{k}$ if $l^{k}_{i}<1$ and $b^{k}_{i}=f\eta^{k}_{i}$ if $l^{k}_{i}\geq1$. We claim that $g^{k}\to f$ in both $h_{X}(\mathbb R^{n})$ and $\mathcal{S}^{\prime}(\mathbb R^{n})$ as $k\to\infty$. In fact, by \cite[Lemma 4.5]{tan2023real}, we have
\[
\begin{aligned}
    \|f-g^{k}\|^{p}_{h_{X}}
    &\leq\left\|  \sum_{i}\mathcal{G}_{N}^{0}b^{k}_{i}  \right\|^{p}_{X}\\
    &\leq\left\|  \sum_{i}\chi_{Q_{i,k}^{*}}\mathcal{G}_{N}f  \right\|^{p}_{X}+C\left\|  \sum_{i}\frac{2^{k}(l^{k}_{i})^{n+d+1}\chi_{(Q^{*}_{i,k})^{c}}}{(l_{i}^{k}+|\cdot-x_{i}|)^{n+d+1}}  \right\|^{p}_{X}\\
    &\leq \left\|  \sum_{i}\chi_{Q_{i,k}^{*}}\mathcal{G}_{N}f  \right\|^{p}_{X}+C\left\|  \sum_{i}2^{k}(\mathcal{M}\chi_{Q^{*}_{i,k}})^{\frac{n+d+1}{n}}  \right\|^{p}_{X}\\
    &\leq C\left\|  \chi_{\Omega_{k}}\mathcal{G}_{N}f  \right\|^{p}_{X}+C\left\|  \sum_{i}2^{k}\chi_{Q^{*}_{i,k}}  \right\|^{p}_{X}\\
    &\leq C\left\|  \chi_{\Omega_{k}}\mathcal{G}_{N}f  \right\|^{p}_{X}+C\left\|  2^{k}\chi_{\Omega_{k}}  \right\|^{p}_{X}\\
    &\leq C\left\|  \chi_{\Omega_{k}}\mathcal{G}_{N}f  \right\|^{p}_{X}.
\end{aligned}
\]
Thus, we obtain that 
\[
\|f-g^{k}\|_{h_X}=\left\|\sum_{i}b^{k}_{i}\right\|_{h_X}\to 0
\]
as $k\to\infty$. Notice that $\|g^{k}\|_{L^{\infty}}\leq C2^{k}$. It follows that $g^{k}\to 0$ uniformly as $k\to\infty$.
Therefore,
\[
f=\sum_{k=-\infty}^{\infty}(g^{k+1}-g^{k})
\]
in $\mathcal{S}^{\prime}(\mathbb R^{n})$ and almost everywhere. In fact, since ${\rm supp}(\sum_{i}b^{k}_{i})\subset\Omega_{k}$, then $g^{k}\to f$ almost everywhere as $k\to\infty$. Then, by applying \cite[Lemma 4.6, Lemma 4.7 and Lemma 4.8]{tan2023real}, we can obtain the desired decomposition as follows:
\[
\begin{aligned}
    g^{k+1}-g^{k}
   & =\left( f-\sum_{j}b^{k+1}_{j}  \right)-\left( 
f-\sum_{i}b^{k}_{i}  \right)\\
   &\equiv\sum_{i}h^{k}_{i},
\end{aligned}
\]
where the series converges both in $\mathcal{S}^{\prime}(\mathbb R^{n})$ and almost everywhere. By repeating the similar argument as in \cite[Theorem 4.3]{tan2023real}, we can categorize $h_{i}^{k}$ into three cases and obtain that ${\rm supp}(h^{k}_{i})\subset\tilde Q^{k}_{i}\subset\Omega_{k}$ and in case 3, $h^{k}_{i}$ satisfies the moment conditions $\int_{\mathbb R^{n}}h^{k}_{i}(x)q(x)dx=0$ for any $q\in\mathcal{P}_{s}(\mathbb R^{n})$ where $\tilde Q^{k}_{i}$ can be chosen in \cite{tan2023real}.\par
Let $\lambda_{i,k}=C2^{k}\left\|\chi_{\tilde Q^{k}_{i}}\right\|_{X}$ and $a_{i,k}=\frac{h^{k}_{i}}{\lambda_{i,k}}$. Then it follows that each $a_{i,k}$ satisfies ${\rm supp}(a_{i,k})\subset\tilde Q^{k}_{i}$, $\|a_{i,k}\|_{L^{q}}\leq\frac{|\tilde Q^{k}_{i}|^{\frac{1}{q}}}{\left\|\chi_{\tilde Q^{k}_{i}}\right\|_{X}}$ and the moment conditions for small cubes with
\[
f=\sum_{i,k}\lambda_{i,k}a_{i,k}.
\]
For convenience, we rearrange $\{a_{i,k}\}$ and $\{\lambda_{i,k}\}$ as follows:
\[
f=\sum_{i,k}\lambda_{i,k}a_{i,k}\equiv\sum_{j=1}^{\infty}\lambda_{j}a_{j}.
\]\par
Then, we need to prove that for any fixed $\eta\in(0,\infty)$,
\[
\left\| \left(\sum_{j=1}^{\infty}\left(\frac{\lambda_{j}\chi_{Q_{j}}}{\|\chi_{Q_{j}}\|_{X}}\right)^{\eta}\right)^{\frac{1}{\eta}}   \right\|_{X}\leq C_{\eta}\|f\|_{h_{X}}.
\]
Observe that
\[
\left\| \left(\sum_{k\in\mathbb Z}\sum_{i\in\mathbb N}\left(\frac{\lambda_{i,k}\chi_{\tilde Q_{i}^{k}}}{\|\chi_{\tilde Q^{k}_{i}}\|_{X}}\right)^{\eta}\right)^{\frac{1}{\eta}}   \right\|_{X}\leq C\left\| \left(\sum_{k\in\mathbb Z}(2^{k}\chi_{\Omega_{k}})^{\eta}\right)^{\frac{1}{\eta}}   \right\|_{X}.
\]
Notice that $\Omega_{k+1}\subset\Omega_{k}$ and $|\bigcap_{k=1}^{\infty}\Omega_{k}|=0$. Then for almost everywhere $x\in\mathbb R^{n}$, we have
\[
\sum_{k=-\infty}^{\infty}(2^{k}\chi_{\Omega_{k}}(x))^{\eta}=\sum_{k=-\infty}^{\infty}2^{k\eta}\sum_{j=k}^{\infty}\chi_{\Omega_{j}\backslash\Omega_{j+1}}(x)=(1-2^{-\eta})^{-1}\sum_{j=-\infty}^{\infty}2^{js}\chi_{\Omega_{j}\backslash\Omega_{j+1}}(x).
\]
Therefore, by the definition of $\Omega_{k}$, we can obtain that
\[
\begin{aligned}
&\left\| \left(\sum_{k=-\infty}^{\infty}(2^{k}\chi_{\Omega_{k}})^{\eta}\right)^{\frac{1}{\eta}}   \right\|_{X}\\
&\leq C_{\eta}\left\| \left(\sum_{j=-\infty}^{\infty}(2^{j}\chi_{\Omega_{j}\backslash\Omega_{j+1}})^{\eta}\right)^{\frac{1}{\eta}}   \right\|_{X}\\
&\leq C_{\eta}\left\| \left(\sum_{j=-\infty}^{\infty}(\mathcal{G}_{N}(f)\chi_{\Omega_{j}\backslash\Omega_{j+1}})^{\eta}\right)^{\frac{1}{\eta}}   \right\|_{X}\\
&\leq C_{\eta}\left\| \mathcal{G}_{N}(f)\sum_{j=-\infty}^{\infty}\chi_{\Omega_{j}\backslash\Omega_{j+1}}   \right\|_{X}\\
&\leq C_{\eta}\left\| \mathcal{G}_{N}(f)   \right\|_{X}.
\end{aligned}
\]
Since $h_{X}(\mathbb R^{n})\cap L^{2}(\mathbb R^{n})$ is dense in $h_{X}(\mathbb R^{n})$, we can prove that any $f\in h_{X}(\mathbb R^{n})$ can be decomposed as in the theorem. Therefore, we complete the proof of the theorem.
\end{proof}

Then, we will discuss the finite decomposition on $h_{X}(\mathbb R^{n})$. Since its proof is similar to that of \cite[Theorem 4.9]{tan2023real}, we only sketch some important steps here.
\begin{definition}\label{def4.4}
    Let $X$ be a ball quasi-Banach function space satisfying Assumption~{\ref{ass2.7}}. Let $p_{0},\ q_{0}$ and $\tau$ be as in Lemma~\ref{le2.8}. Assume that $X^{1/p_{0}}$ is a ball Banach function space and the Hardy--Littlewood maximal operator $\mathcal{M}$ is bounded on $[(X^{1/p_{0}})^{\prime}]^{1/\tau^{\prime}}$. Let $1<q\leq\infty$ and $d\in\mathbb Z_{+}$. The \emph{finite atomic local Hardy space} $h_{fin}^{X,q,d}(\mathbb R^{n})$ associated with $X$, is defined by
    \[
    h_{fin}^{X,q,d}(\mathbb R^{n})=\left\{ f\in\mathcal{S}^{\prime}(\mathbb R^{n})\colon f=\sum_{j=1}^{M}\lambda_{j}a_{j} \right\},
    \]
    where $\{a_{j}\}_{j=1}^{M}$ are local-$(X,q,d)$-atoms satisfying
    \[
    \left\| \sum_{j=1}^{M} \frac{\lambda_{j}\chi_{Q_{j}}}{\|\chi_{Q_{j}}\|_{X}}   \right\|_{X}<\infty.
    \]
    Furthermore, the quasi-norm $\|\cdot\|_{h_{fin}^{X,q,d}}$ in $h_{fin}^{X,q,d}(\mathbb R^{n})$ is defined by setting, for any $f\in h_{fin}^{X,q,d}(\mathbb R^{n})$,
    \[
    \|f\|_{h_{fin}^{X,q,d}}:=\inf\left\{\left\| \sum_{j=1}^{M} \frac{\lambda_{j}\chi_{Q_{j}}}{\|\chi_{Q_{j}}\|_{X}}   \right\|_{X}\right\},
    \]
    where the infimum is taken over all finite decomposition of $f$.
\end{definition}
From Corollary~\ref{cor4.3}, we know that $h_{fin}^{X,q,d}(\mathbb R^{n})$ is dense in $h_{X}(\mathbb R^{n})$.
\begin{theorem}\label{th4.5}
    Let $X$ be a ball quasi-Banach function space satisfying Assumption~{\ref{ass2.7}}. Let $p_{0},\ q_{0}$ and $\tau$ be as in Lemma~\ref{le2.8}. Assume that $X^{1/p_{0}}$ is a ball Banach function space and the Hardy--Littlewood maximal operator $\mathcal{M}$ is bounded on $[(X^{1/p_{0}})^{\prime}]^{1/\tau^{\prime}}$.
    Further assume that $X$ has an absolutely continuous quasi-norm. Fix $q\in(1,\infty)$ and $d\in\mathbb{Z}_{+}$. For $f\in h_{fin}^{X,q,d}(\mathbb R^{n})$,
    \[
    \|f\|_{h_{fin}^{X,q,d}}\sim\|f\|_{h_{X}}.
    \]
\end{theorem}\par
\begin{remark}\label{re4.6}
    From \cite[Theorem 1.10]{yan2020intrinsic} and Theorem~\ref{th4.5}, we obtain the following conclusion:\par
    Let $X,q,d$ be as Theorem~\ref{th4.5}.
    \begin{enumerate}[(i)]
        \item If $q\in(1,\infty)$, then $\|\cdot\|_{h^{X,q,d}_{fin}}$ and $\|\cdot\|_{h_{X}}$ are equivalent quasi-norms on the space $h^{X,q,d}_{fin}(\mathbb R^{n})$.
        \item If $q=\infty$, then $\|\cdot\|_{h^{X,\infty,d}_{fin}}$ and $\|\cdot\|_{h_{X}}$ are equivalent quasi-norms on the space $h^{X,q,d}_{fin}(\mathbb R^{n})\cap\mathscr{C}(\mathbb R^{n})$, where $\mathscr{C}(\mathbb R^{n})$ is defined to be the set of all continuous complex-valued functions on $\mathbb R^{n}$.
    \end{enumerate}
\end{remark}
\begin{proof}[{\bf Proof of Theorem~{\upshape\ref{th4.5}}}]
    First, it is obvious that for $f\in h^{X,q,d}_{fin}(\mathbb R^{n})$
\[
\|f\|_{h_{X}}\leq C\|f\|_{h^{X,q,d}_{fin}}.
\]
Then we only need to prove that for any $f\in h^{X,q,d}_{fin}(\mathbb R^{n})$,
\[
\|f\|_{h^{X,q,d}_{fin}}\leq C\|f\|_{h_{X}}.
\]
By homogeneity, we can assume that $\|f\|_{h_{X}}=1$. Thus, it suffices to show that $\|f\|_{h^{X,q,d}_{fin}}\leq C$. By Theorem~{\ref{th4.2}}, since $f\in h_{X}(\mathbb R^{n})\cap L^{q}(\mathbb R^{n})$, we can form the following decomposition of $f$ in terms of the local-$(X,q,d)$-atoms:
\[
f=\sum_{i,k}\lambda_{i,k}a_{i,k},
\]
where the series converges almost everywhere and in $\mathcal{S}^{\prime}(\mathbb R^{n})$. If $f\in h^{X,q,d}_{fin}(\mathbb R^{n})$, then ${\rm supp}f\subset Q(x_{0},R_{0})$ for fixed $x_{0}\in\mathbb R^{n}$ and some $R_{0}\in(1,\infty)$. Set $\hat{Q}_{0}=Q(x_{0},\sqrt{n}R_{0}+2^{3(n+10)+1})$. Then for any $\psi\in\mathcal{S}_{N,2^{3(n+10)}}$ and $x\in(\hat{Q}_{0})^{c}$, for $0<t<1$, we obtain that
\[
\psi_{t}\ast f(x)=\int_{Q(x_{0},R_{0})}\psi_{t}(x-y)f(y)dy=0.
\]
Thus, for any $x\in(\hat{Q}_{0})^{c}$, it follows that $x\in(\Omega_{k})^{c}$. Hence, $\Omega_{k}\subset\hat{Q}_{0}$ and ${\rm supp}\sum_{i,k}\lambda_{i,k}a_{i,k}\subset\hat{Q}_{0}$. Next we claim that $\sum_{i,k}\lambda_{i,k}a_{i,k}$ converges to $f$ in $L^{q}(\mathbb R^{n})$. In fact, for any $x\in\mathbb R^{n}$, we can find a $j\in\mathbb Z$ such that $x\in\Omega_{j}\backslash\Omega_{j+1}$. Since ${\rm supp}a_{i,k}\subset\tilde Q_{i}^{k}\subset\Omega_{k}\subset\Omega_{j+1}$ for all $k>j$, then we obtain that
\[
\left| \sum_{k\in\mathbb Z}\sum_{i\in\mathbb N}\lambda_{i,k}a_{i,k}  \right|\leq\sum_{k\in\mathbb Z}\sum_{i\in\mathbb N}\left| \lambda_{i,k}a_{i,k}  \right|\leq C\sum_{k\leq j}2^{k}\leq C2^{j}\leq C\mathcal{G}_{N}f.
\]
By applying the fact that $\mathcal{G}_{N}f\in L^{q}(\mathbb R^{n})$ and the Lebesgue dominated convergence theorem, we completed the proof of the claim. For each integer $N>0$, we write
\[
F_{N}=\{(i,k)\colon k\in\mathbb Z, i\in\mathbb N, |k|+i\leq N\}.
\]
Then $f_{N}\equiv\sum_{(i,k)\in F_{N}}\lambda_{i,k}a_{i,k}$ is a finite combination of the local-$(X,q,d)$-atoms with
\[
\|f_{N}\|_{h^{X,q,d}_{fin}}\leq\left\|  \sum_{(i,k)\in F_{N}}\frac{\lambda_{i,k}\chi_{\tilde Q_{i}^{k}}}{\|\chi_{\tilde Q_{i}^{k}}\|_{X}}  \right\|_{X}\leq C\|f\|_{h_{X}}\leq C.
\]
Since the series $\sum_{i,k}\lambda_{i,k}a_{i,k}$ converges absolutely in $L^{q}(\mathbb R^{n})$, for any given $\epsilon\in(0,\infty)$, there exists $N$ such that $\|f-f_{N}\|_{L^{q}}<\frac{\epsilon|\hat{Q}_{0}|^{1/q}}{\left\|\chi_{\hat{Q}_{0}}\right\|_{X}}$. Meanwhile, ${\rm supp}f_{N}\subset\hat{Q}_{0}$ and the support of $f$ imply that ${\rm supp}(f-f_{N})\subset\hat{Q}_{0}$. Therefore, we can divide $\hat{Q}_{0}$ into the union of cube of $\{Q_{i}\}_{i=1}^{N_{0}}$ with disjoint interior and sidelengths satisfying $l_{i}\in[1,2)$, where $N_{0}$ depends only on $R_{0}$ and $n$. In particular, we know that $g_{N,i}\equiv\frac{(f-f_{N})\chi_{Q_{i}}}{\epsilon}$ is a local-$(X,q,d)$-atom. Therefore,
\[
\begin{aligned}
\|f-f_{N}\|_{h^{X,q,d}_{fin}}
&=\left\| \sum_{i=1}^{N_{0}}\epsilon g_{N,i} \right\|_{h^{X,q,d}_{fin}}\\
&\leq C\left\|  \sum_{i=1}^{N_{0}}\frac{\epsilon\chi_{Q_{i}}}{\|\chi_{Q_{i}}\|_{X}} \right\|_{X}\\
&\leq C\sum_{i=1}^{N_{0}}\left\|  \frac{\epsilon\chi_{Q_{i}}}{\|\chi_{Q_{i}}\|_{X}} \right\|_{X}\\
&\leq C.
\end{aligned}
\]
Thus, we conclude that
\[
f=\sum_{(i,k)\in F_{N}}\lambda_{i,k}a_{i,k}+\sum_{i=1}^{N_{0}}\epsilon g_{N,i}
\]
is the desired finite atomic decomposition. Therefore, we finish the proof of Theorem~{\ref{th4.5}}.
\end{proof}

\section{Dual spaces}\label{se5}
As an application of the above atomic decompositions, we will give the dual space of $h_{X}(\mathbb R^{n})$. First, we introduce the local ball Campanato-type function space. For any bounded measurable set $E$ with positive measure and any locally integrable function $f$ on $\mathbb R^{n}$, $P_{E}^{(s)}(f)$ denotes the \emph{minimizing polynomial} of $f$ with total degree not greater than $s$, which means that $P_{E}^{(s)}(f)$ is the unique polynomial in $\mathcal{P}_{s}(\mathbb R^{n})$ such that, for any $P\in\mathcal{P}_{s}(\mathbb R^{n})$,
\[
\int_{E}\left[ f(x)-P_{E}^{(s)}(f)(x) \right]P(x)dx=0.
\]
\begin{definition}\label{def5.1}
    Let $X$ be a ball quasi-Banach function space. Let $1<q<\infty$ and $d\in\mathbb Z_{+}$. The \emph{local ball Campanato-type function space} $\widetilde{bmo}^{X,q,d}(\mathbb R^{n})$, associated with $X$, is defined to be the set of all $f\in L^{q}_{loc}(\mathbb R^{n})$ such that
    \[
    \begin{aligned}
    \|f\|_{\widetilde{bmo}^{X,q,d}}=
    &\sup_{\vert Q\vert\geq1}\left\|  \sum_{i=1}^{M}\frac{\lambda_{i}\chi_{Q_{i}}}{\|\chi_{Q_{i}}\|_{X}} \right\|_{X}^{-1}\sum_{i=1}^{M}\left\{ \frac{\lambda_{i}\vert Q_{i}\vert}{\|\chi_{Q_{i}}\|_{X}}\left( \frac{1}{\vert Q_{i}\vert}\int_{Q_{i}}\vert f(x)\vert^{q}dx \right)^{\frac{1}{q}} \right\}\\
    &+\sup_{\vert Q\vert<1}\left\| \sum_{i=1}^{M}\frac{\lambda_{i}\chi_{Q_{i}}}{\|\chi_{Q_{i}}\|_{X}} \right\|_{X}^{-1}\sum_{i=1}^{M}\left\{\frac{\lambda_{i}\vert Q_{i}\vert}{\|\chi_{Q_{i}}\|_{X}}\right.\\
    &\left.\times\left( \frac{1}{\vert Q_{i}\vert}\int_{Q_{i}}\vert f(x)-P^{(d)}_{Q_{i}}f(x) \vert^{q}dx \right)^{\frac{1}{q}}\right\}.
    \end{aligned}
    \]
\end{definition}
\begin{theorem}\label{th5.2}
    Let $X$ be a ball quasi-Banach function space satisfying Assumption~{\ref{ass2.7}}. Let $p_{0},\ q_{0}$ and $\tau$ be as in Lemma~\ref{le2.8}. Assume that $X^{1/p_{0}}$ is a ball Banach function space and the Hardy--Littlewood maximal operator $\mathcal{M}$ is bounded on $[(X^{1/p_{0}})^{\prime}]^{1/\tau^{\prime}}$.
    Further assume that $X$ has an absolutely continuous quasi-norm. Let $1<q<\infty$ and $d\in\mathbb Z_{+}$. The dual space of $h_{X}(\mathbb R^{n})$, denoted by $\left(h_{X}(\mathbb R^{n})\right)^{*}$, is $\widetilde{bmo}^{X,q^{\prime},d}(\mathbb R^{n})$ with $\frac{1}{q}+\frac{1}{q^{\prime}}=1$ in the following sense:
    \begin{enumerate}
        \item Let $g\in\widetilde{bmo}^{X,q^{\prime},d}(\mathbb R^{n})$. Then the linear functional 
        \[
        L_{g}\colon f\to L_{g}(f):=\int_{\mathbb R^{n}}f(x)g(x)dx,
        \]
        initially defined by any $f\in h_{fin}^{X,q,d}(\mathbb R^{n})$, has a bounded extension to $h_{X}(\mathbb R^{n})$.
        \item Conversely, for any $L\in(h_{X}(\mathbb R^{n}))^{*}$, there exists a unique function $g\in\widetilde{bmo}^{X,q^{\prime},d}$ such that $L(f)=\langle g,f\rangle$ holds for any $f\in h^{X,q,d}_{fin}(\mathbb R^{n})$ with $\|g\|_{\widetilde{bmo}^{X,q^{\prime},d}}\leq C\|L\|$.
    \end{enumerate}
\end{theorem}
\begin{proof}[{\bf Proof of Theorem~{\upshape\ref{th5.2}}}]
First we prove (1). From Theorem~{\ref{th4.5}}, we can know that for any $f\in h_{fin}^{X,q,d}(\mathbb R^{n})$, $f$ can be decomposed into $f=\sum_{j=1}^{M_1}\lambda_{j}a_{j}+\sum_{j=1}^{M_2}\mu_{j}b_{j}$, where every $a_{j}$ is supported in $Q_{j}$ with $\vert Q_{j}\vert<1$ and every $b_{j}$ is supported in $P_{j}$ with $\vert P_{j}\vert\geq1$.
For the sum $\sum_{j=1}^{M_1}\lambda_{j}a_{j}$ and $g\in\widetilde{bmo}^{X,q^{\prime},d}(\mathbb R^{n})$,
\[
\begin{aligned}
    \left\vert L_{g}\left(\sum_{j=1}^{M_1}\lambda_{j}a_{j}\right)\right\vert
    &=\left\vert \int_{\mathbb R^{n}}\sum_{j=1}^{M_1}\lambda_{j}a_{j}(x)g(x)dx \right\vert\\
    &\leq\sum_{j=1}^{M}\lambda_{j}\left\vert\int_{\mathbb R^{n}}a_{j}(x)[g(x)-P_{Q_{j}}^{(d)}g(x)]dx\right\vert\\
    &\leq\sum_{j=1}^{M_{1}}\lambda_{j}\|a_{j}\|_{L^{q}}\|g-P_{Q_{j}}^{(d)}g\|_{L^{q^{\prime}}}\\
    &\leq\sum_{j=1}^{M_{1}}\frac{\lambda_{j}\vert Q_{j}\vert}{\|\chi_{Q_{j}}\|_{X}}\left( \frac{1}{\vert Q_{j}\vert^{q^{\prime}}}\int_{\mathbb R^{n}}\vert g(x)-P_{Q_{j}}^{(d)}g(x)\vert^{q^{\prime}}dx \right)^{\frac{1}{q^{\prime}}}\\
    &\leq\left\| \sum_{j=1}^{M_{1}}\frac{\lambda_{j}\chi_{Q_{j}}}{\|\chi_{Q_{j}}\|_{X}} \right\|_{X}\|g\|_{\widetilde{bmo}^{X,q^{\prime},d}}\\
    &\leq C\|f\|_{h_{fin}^{X,q,d}}\|g\|_{\widetilde{bmo}^{X,q^{\prime},d}}\\
    &\sim\|f\|_{h_{X}}\|g\|_{\widetilde{bmo}^{X,q^{\prime},d}}.
\end{aligned}
\]
For the sum $\sum_{j=1}^{M_2}\mu_{j}b_{j}$,
\[
\begin{aligned}
    \left\vert L_{g}\left(\sum_{j=1}^{M_2}\mu_{j}b_{j}\right)\right\vert
    &=\left\vert \int_{\mathbb R^{n}}\sum_{j=1}^{M_2}\mu_{j}b_{j}(x)g(x)dx \right\vert\\
    &\leq\sum_{j=1}^{M_{2}}\mu_{j}\left\vert \int_{\mathbb R^{n}}b_{j}(x)g(x)dx \right\vert\\
    &\leq\sum_{j=1}^{M_{2}}\mu_{j}\|b_{j}\|_{L^{q}}\|g\|_{L^{q^{\prime}}}\\
    &\leq\sum_{j=1}^{M_{2}}\frac{\mu_{j}\vert P_{j}\vert}{\|\chi_{P_{j}}\|_{X}}\left(  \frac{1}{\vert P_{j}\vert^{q^{\prime}}}\int_{\mathbb R^{n}}\vert g(x)\vert^{q^{\prime}}dx \right)^{\frac{1}{q^{\prime}}}\\
    &\leq\left\| \sum_{j=1}^{M_{2}}\frac{\mu_{j}\chi_{P_{j}}}{\|\chi_{P_{j}}\|_{X}} \right\|_{X}\|g\|_{\widetilde{bmo}^{X,q^{\prime},d}}\\
    &\leq C\|f\|_{h_{fin}^{X,q,d}}\|g\|_{\widetilde{bmo}^{X,q^{\prime},d}}\\
    &\sim\|f\|_{h_{X}}\|g\|_{\widetilde{bmo}^{X,q^{\prime},d}}.    
\end{aligned}
\]
Therefore, by the density, we conclude that (1) holds true. It remains to prove (2). Fix a cube $P$ with $l(P)\geq1$. For any given $f\in L^{q}(P)$ with $\|f\|_{L^{q}(P)}>0$, set
\[
a(x)\equiv\frac{\vert P\vert^{\frac{1}{q}}f(x)\chi_{P}(x)}{\|f\|_{L^{q}(P)}\|\chi_{P}\|_{X}}.
\]
Obviously, $a$ is a local-$(X,q,d)$-atom. For any $L\in(h_{X}(\mathbb R^{n}))^{*}$,
\[
\vert L(a)\vert\leq\|L\|\|a\|_{h_{X}}\leq C\|L\|.
\]
Thus, we can obtain that
\[
\vert L(f)\vert\leq\|L\|\|f\|_{h_{X}}\leq C\|L\|\|f\|_{L^{q}(P)}\vert P\vert^{-\frac{1}{q}}\|\chi_{P}\|_{X},
\]
which implies that $L\in(L^{q}(P))^{*}$ and $(h_{X}(\mathbb R^{n}))^{*}\subset(L^{q}(P))^{*}$. Since $1<q<\infty$, by using the duality $L^{q}(P)$--$L^{q^{\prime}}(P)$, we find that there exists a $g^{P}\in L^{q^{\prime}}(P)$ such that for all $f\in L^{q}(P)$,
\[
L(f)=\int_{P}f(x)g^{P}(x)dx,
\]
and $\|g^{P}\|_{L^{q^{\prime}}(P)}\leq C\|L\|\vert P\vert^{-\frac{1}{q}}\|\chi_{P}\|_{X}$. Take a sequence $\{P_{j}\}_{j\in\mathbb N}$ of cubes such that $P_{j}\subset P_{j+1}$, $\bigcup_{j\in\mathbb N}P_{j}=\mathbb R^{n}$ and $l(P_{1})\geq1$. Thus, by repeating the similar arguments, we can know that there exists a $g^{P_{j}}\in L^{q^{\prime}}(P_{j})$ such that for each $P_{j}$ and any $f\in L^{q}(P_{j})$,
\[
L(f)=\int_{P_{j}}f(x)g^{P_{j}}(x)dx,
\]
and $\|g^{P_{j}}\|_{L^{q^{\prime}}(P_{j})}\leq C\|L\|\vert P_{j}\vert^{-\frac{1}{q}}\|\chi_{P_{j}}\|_{X}$. Then we can construct a function $g$ such that for all $f\in L^{q}(P_{j})$ and all $j\in\mathbb N$,
\[
L(f)=\int_{P_{j}}f(x)g(x)dx.
\]
For more details, see \cite{tan2023real}. Thus, we conclude that
\[
L(f)=\langle g,f \rangle=\int_{\mathbb R^{n}}f(x)g(x)dx
\]
holds true for all $f\in h_{fin}^{X,q,d}(\mathbb R^{n})$.\par
When $\vert Q_{j}\vert<1$, let $f_{j}\in L^{q}(Q_{j})$ with $\|f_{j}\|_{L^{q}(Q_{j})}=1$ satisfying
\[
\left[  \int_{Q_{j}}\vert g(x)-P_{Q_{j}}^{(d)}g(x)\vert^{q^{\prime}}dx \right]^{\frac{1}{q^{\prime}}}=\int_{Q_{j}}[g(x)-P_{Q_{j}}^{(d)}g(x)]f_{j}(x)dx
\]
and, for any $x\in\mathbb R^{n}$, define
\[
a_{j}(x)\equiv\frac{\vert Q_{j}\vert^{\frac{1}{q}}(f_{j}(x)-P_{Q_{j}}^{(d)}f_{j}(x))\chi_{Q_{j}}}{\|f_{j}-P_{Q_{j}}^{(d)}f_{j}(x)\|_{L^{q}(Q_{j})}\|\chi_{Q_{j}}\|_{X}}.
\]
Obviously, $a_{j}$ is a local-$(X,q,d)$-atom. Then, if $L\in(h_{X}(\mathbb R^{n}))^{*}$, we can get that
\[
L\left(\sum_{j=1}^{M}\lambda_{j}a_{j}\right)\leq\|L\|\left\| \sum_{j=1}^{M}\lambda_{j}a_{j} \right\|_{h_{X}}\leq C\|L\|\left\|  \sum_{j=1}^{M}\frac{\lambda_{j}\chi_{Q_{j}}}{\|\chi_{Q_{j}}\|_{X}} \right\|_{X}.
\]
Moreover, we have
\[
\begin{aligned}
    &\sum_{j=1}^{M}\frac{\lambda_{j}\vert Q_{j}\vert}{\|\chi_{Q_{j}}\|_{X}}\left( \frac{1}{\vert Q_{j}\vert}\int_{Q_{j}}\vert g(x)-P_{Q_{j}}^{(d)}g(x)\vert^{q^{\prime}}dx \right)^{\frac{1}{q^{\prime}}}\\
    &=\sum_{j=1}^{M}\frac{\lambda_{j}}{\|\chi_{Q_{j}}\|_{X}}\vert Q_{j}\vert^{\frac{1}{q}}\int_{Q_{j}} [g(x)-P_{Q_{j}}^{(d)}g(x)]f_{j}(x)dx \\
    &=\sum_{j=1}^{M}\frac{\lambda_{j}\vert Q_{j}\vert^{\frac{1}{q}}}{\|\chi_{Q_{j}}\|_{X}}\int_{Q_{j}} [f_{j}(x)-P_{Q_{j}}^{(d)}f_{j}(x)]g(x)\chi_{Q_{j}}dx\\
    &=\sum_{j=1}^{M}\frac{\lambda_{j}\vert Q_{j}\vert^{\frac{1}{q}}}{\|\chi_{Q_{j}}\|_{X}}\|f_{j}-P_{Q_{j}}^{(d)}f_{j}(x)\|_{L^{q}(Q_{j})}\|\chi_{Q_{j}}\|_{X}\vert Q_{j}\vert^{-\frac{1}{q}}\int_{Q_{j}}a_{j}(x)g(x)dx.
\end{aligned}
\]
Then, by using the fact that $\|P_{Q_{j}}^{(d)}(f_{j})\|_{L^{q}}\leq C\|f_{j}\|_{L^{q}}$, we conclude that
\[
\begin{aligned}
    &\sum_{j=1}^{M}\frac{\lambda_{j}\vert Q_{j}\vert}{\|\chi_{Q_{j}}\|_{X}}\left( \frac{1}{\vert Q_{j}\vert}\int_{Q_{j}}\vert g(x)-P_{Q_{j}}^{(d)}g(x)\vert^{q^{\prime}}dx \right)^{\frac{1}{q^{\prime}}}\\
    &\leq C \sum_{j=1}^{M}\lambda_{j}\int_{Q}a_{j}(x)g(x)dx\\
    &\sim\sum_{j=1}^{M_{1}}\lambda_{j}L(a_{j})\sim L\left( 
    \sum_{j=1}^{M}\lambda_{j}a_{j} \right)\\
    &\leq C\|L\|\left\| \sum_{j=1}^{M}\frac{\lambda_{j}\chi_{Q_{j}}}{\|\chi_{Q_{j}}\|_{X}} \right\|_{X}.
\end{aligned}
\]\par
When $\vert Q_{j}\vert\geq1$, let $f_{j}\in L^{q}(Q_{j})$ with $\|f_{j}\|_{L^{q}(Q_{j})}=1$ satisfying
\[
\left(\int_{Q_{j}}\vert g(x)\vert^{q^{\prime}}dx\right)^{\frac{1}{q^{\prime}}}=\int_{Q_{j}}f_{j}(x)g(x)dx
\]
and, for any $x\in\mathbb R^{n}$, define
\[
b_{j}(x)=\frac{\vert Q_{j}\vert^{\frac{1}{q}}f_{j}(x)\chi_{Q_{j}}}{\|f_{j}\|_{L^{q}(Q_{j})}\|\chi_{Q_{j}}\|_{X}}.
\]
Obviously, $a_{j}$ is a local-$(X,q,d)$-atom. Then, we can obtain that
\[
\begin{aligned}
    &\sum_{j=1}^{M_{2}}\frac{\mu_{j}\vert Q_{j}\vert}{\|\chi_{Q_{j}}\|_{X}}\left( \frac{1}{\vert Q_{j}\vert}\int_{Q_{j}}\vert g(x)\vert^{q^{\prime}}dx \right)^{\frac{1}{q^{\prime}}}\\
    &\leq\sum_{j=1}^{M_{2}}\frac{\mu_{j}\vert Q_{j}\vert^{\frac{1}{q}}}{\|\chi_{Q_{j}}\|_{X}}\vert Q_{j}\vert^{-\frac{1}{q}}\|f_{j}\|_{L^{q}(Q_{j})}\|\chi_{Q_{j}}\|_{X}\int_{Q_{j}}b_{j}(x)g(x)dx\\
    &=\sum_{j=1}^{M_{2}}\mu_{j}\int_{Q_{j}}b_{j}(x)g(x)dx\\
    &\sim\sum_{j=1}^{M_{2}}\mu_{j}L(b_{j})\sim L(\sum_{j=1}^{M_{2}}\mu_{j}b_{j})\\
    &\leq C\|L\|\left\|  \sum_{j=1}^{M_{2}}\frac{\mu_{j}\chi_{Q_{j}}}{\|\chi_{Q_{j}}\|_{X}} \right\|_{X}.
\end{aligned}
\]
Therefore, we conclude that (2) holds true.
\end{proof}

\section{Applications}\label{se6}
In this section, we apply Theorems~\ref{th4.1},\ \ref{th4.2},\ \ref{th4.5},\ \ref{th5.2} and Corollary~\ref{cor4.3} to some concrete examples of ball quasi-Banach function spaces, which implies further applications of the main results.
\subsection{Lebesgue spaces and weighted Lebesgue spaces}
First, we apply the obtained results to some fundamental spaces, namely, the Lebesgue space and the weighted Lebesgue space.\par
Let $X:=L^{p}(\mathbb R^{n})$ with $p\in(0,\infty)$. By \cite[Theorem 1]{fefferman1971some}, we know that Assumption~\ref{ass2.7} holds true when $s\in(0,1]$ and $\theta\in(0,\min\{s,p\})$. Besides, the hypothesis of Lemma~\ref{le2.8} holds true when $p_{0}\in(0,p)$ and $q_{0}\in(\max\{1,p\},\infty]$. Moreover, $L^{p}(\mathbb R^{n})$ has an absolutely continuous quasi-norm. Therefore, Theorem~\ref{th4.1},\ \ref{th4.2},\ \ref{th4.5},\ \ref{th5.2} and Corollary~\ref{cor4.3} hold true with X replaced by $L^{p}(\mathbb R^{n})$.\par
In fact, in \cite[Subsection 7.1]{sawano2017hardy} the authors have shown that $L^{p}_{\omega}(\mathbb R^{n})$ with $p\in(0,\infty)$ and $\omega\in A_{\infty}(\mathbb R^{n})$ is a ball quasi-Banach function space. Moreover, by \cite[Theorem 3.1(b)]{andersen1981weighted}, Assumption~\ref{ass2.7} holds true when $\theta,\ s\in(0,1]$, $\theta<s$ and $X:=L^{p}_{\omega}(\mathbb R^{n})$ with $p\in(\theta,\infty)$ and $\omega\in A_{p/\theta}(\mathbb R^{n})$. Besides, let $X:=L^{p}_{\omega}(\mathbb R^{n})$ with $p\in(0,\infty)$ and $\omega\in A_{\infty}(\mathbb R^{n})$. From \cite[Theorem 7.3]{2001fourier}, the hypothesis of Lemma~\ref{le2.8} holds true for $p_{0}\in(0,p)$, $\omega\in A_{p/p_{0}}$ and $q_{0}\in(\max\{1,p\},\infty]$ large enough such that $\omega^{1-(p/p_{0})^{\prime}}\in A_{(p/p_{0})^{\prime}/(q_{0}/p_{0})^{\prime}}$. Moreover. $L^{p}_{\omega}(\mathbb R^{n})$ has an absolutely continuous quasi-norm. Therefore, Theorems~\ref{th4.1},\ \ref{th4.2},\ \ref{th4.5},\ \ref{th5.2} and Corollary~\ref{cor4.3} hold true with X replaced by $L^{p}_{\omega}(\mathbb R^{n})$.

\subsection{Variable Lebesgue space}

Let $\mathcal{P}_{0}$ be the collection of all measurable functions $p(\cdot)\colon\mathbb R^{n}\to(0,\infty)$. For any $p(\cdot)\in\mathcal{P}_{0}$, let 
    \[
    {p}_{+}:=\operatorname*{ess\,sup}\limits_{x\in\mathbb R^{n}}p(x)\quad {\rm and} \quad {p}_{-}:=\operatorname*{ess\,inf}_{x\in\mathbb R^{n}}p(x).
    \]
    A function $p(\cdot)\in\mathcal{P}_{0}$ is said to be \emph{globally log-H\"older continuous}, denoted by $p(\cdot)\in LH$ if there exists a positive constant $p_{\infty}$ such that, for any $x,\ y\in\mathbb R^{n}$,
    \[
    \vert p(x)-p(y) \vert\lesssim\frac{1}{-\log(\vert x-y \vert)},\quad \vert x-y \vert<\frac{1}{2}
    \]
    and
    \[
    \vert p(x)-p_{\infty} \vert\lesssim\frac{1}{\log(\vert x \vert+e)}.
    \]
\begin{definition}
    Let $p(\cdot)\colon\mathbb R^{n}\to[0,\infty)$ be a measurable function. The \emph{variable Lebesgue space} $L^{p(\cdot)}(\mathbb R^{n})$ is defined as the set of all measurable functions $f$ for which the quantity $\int_{\mathbb R^{n}}\vert\epsilon f(x)\vert^{p(x)}dx$ is finite for some $\epsilon>0$ and 
    \[
    \|f\|_{L^{p(\cdot)}}:=\inf\left\{  \lambda>0\colon\int_{\mathbb R^{n}}\left( \frac{\vert f(x)\vert}{\lambda} \right)^{p(x)}dx\leq1 \right\}.
    \]
\end{definition}
    In fact, from \cite[Subsection 7.4]{sawano2017hardy} we know that whenever $p(\cdot)\in\mathcal{P}_{0}$, $L^{p(\cdot)}(\mathbb R^{n})$ is a ball quasi-Banach function space. Let $X:=L^{p(\cdot)}(\mathbb R^{n})$ with $p(\cdot)\in LH$. By \cite{cruz2014variable,cruz2020}, we know that Assumption~\ref{ass2.7} holds true when $s\in(0,1]$ and $\theta\in(0,\min\{s,p_{-}\})$. Besides, let $X:=L^{p(\cdot)}(\mathbb R^{n})$ with $p(\cdot)\in LH$ and $0<p_{-}\leq p_{+}<\infty$. From \cite[Theorem 3.16]{cruz2013variable}, the hypothesis of Lemma~\ref{le2.8} holds true when $p_{0}\in(0,p_{-})$ and $q_{0}\in(\max\{1,p_{+}\},\infty]$. Moreover. $L^{p(\cdot)}(\mathbb R^{n})$ has an absolutely continuous quasi-norm. Therefore, Theorems~\ref{th4.1},\ \ref{th4.2},\ \ref{th4.5},\ \ref{th5.2} and Corollary~\ref{cor4.3} hold true with X replaced by $L^{p(\cdot)}(\mathbb R^{n})$. We remark that these results have been obtained in \cite{tan2023real}.

\subsection{Lorentz spaces}
The Lorentz space was first introduced by Lorentz in \cite{Lorentz1951OnTT}. We refer the reader to \cite{grafakos2014classical,sawano2017hardy,Sawyer1990BoundednessOC} for more studies on Lorentz spaces.
\begin{definition}
    The \emph{Lorentz space} $L^{p,q}(\mathbb R^{n})$ is defined to be the set of all measurable functions $f$ on $\mathbb R^{n}$ such that, when $p,q\in(0,\infty)$,
    \[
    \|f\|_{L^{p,q}}:=\left\{  \int_{0}^{\infty}\left[ t^{\frac{1}{p}}f^{*}(t) \right]^{q}\frac{dt}{t} \right\}^{\frac{1}{q}}<\infty,
    \]
    and, when $p\in(0,\infty)$ and $q=\infty$,
    \[
    \|f\|_{L^{p,q}}:=\sup_{t\in(0,\infty)}t^{\frac{1}{p}}f^{*}(t)<\infty,
    \]
    where $f^{*}$ denotes the decreasing rearrangement of $f$, which is defined by setting, for any $t\in[0,\infty)$,
    \[
    f^{*}(t):=\inf\{s\in(0,\infty)\colon \mu_{f}(s)\leq t\}
    \]
    with $\mu_{f}(s):=\left\vert  \left\{  x\in\mathbb R^{n}\colon\vert f(x)\vert>s \right\} \right\vert$.
\end{definition}
From \cite{sawano2017hardy}, we can know that when $p,q\in(1,\infty)$ or $p\in(1,\infty)$ and $q=\infty$, $L^{p,q}(\mathbb R^{n})$ is a ball Banach function space; when $p,q\in(0,\infty)$ or $p\in(0,\infty)$ and $q=\infty$, $L^{p,q}(\mathbb R^{n})$ is a ball quasi-Banach function space. Moreover, from \cite{carro2007}, we can know that $L^{p,q}(\mathbb R^{n})$ has an absolutely continuous norm. 
Let $X:=L^{p,q}(\mathbb R^{n})$ with $p\in(0,\infty)$ and $q\in(0,\infty]$. From \cite[Theorem 2.3(iii)]{curbera2006}, Assumption~\ref{ass2.7} holds true when  $s\in(0,1]$ and $\theta\in(0,\min\{s,p\})$. Besides, let $X:=L^{p,q}(\mathbb R^{n})$ with $p\in(0,\infty)$ and $q\in(0,\infty)$. By \cite[Theorem 1.4.16]{grafakos2014classical}, the hypothesis of Lemma~\ref{le2.8} holds true when $p_{0}\in(0,\min\{p,q\})$ and $q_{0}\in(\max\{1,p,q\},\infty]$. Therefore, Theorems~\ref{th4.1},\ \ref{th4.2},\ \ref{th4.5},\ \ref{th5.2} and Corollary~\ref{cor4.3} hold true with X replaced by $L^{p,q}(\mathbb R^{n})$.

\subsection{Mixed-norm Lebesgue space}
As a natural generalization of the Lebesgue space $L^{p}(\mathbb R^{n})$, the study of the mixed-norm Lebesgue space $L^{\Vec{p}}(\mathbb R^{n})$ can be traced back to H\"ormander \cite{hormander1960} and was further developed by Benedek and Panzone \cite{10.1215/S0012-7094-61-02828-9}. For more information on mixed-norm type spaces, see \cite{clean2017,huang2019,huang2021}.\par
\begin{definition}
    Let $\Vec{p}:=(p_{1}, \cdots,p_{n})\in(0,\infty]^{n}$. The \emph{mixed-norm Lebesgue space} $L^{\Vec{p}}(\mathbb R^{n})$ is defined to be the set of all the measurable functions $f$ on $\mathbb R^{n}$ such that
    \[
    \|f\|_{L^{\Vec{p}}}:=\left\{  \int_{\mathbb R^{n}}\cdots\left[ \int_{\mathbb R^{n}}\vert f(x_{1},\cdots,x_{n})\vert^{p_{1}}dx_{1}  \right]^{\frac{p_{2}}{p_{1}}}\cdots dx_{n}\right\}^{\frac{1}{p_{n}}}<\infty
    \]
    with the usual modifications made when $p_{i}=\infty$ for some $i\in\{1,\cdots,n\}$. 
\end{definition}
    For any exponent vector $\Vec{p}:=(p_{1}, \cdots,p_{n})$, let
    \[
    p_{-}:=\min\{p_{1},\cdots,p_{n}\}\quad{\rm and}\quad p_{+}:=\max\{p_{1},\cdots,p_{n}\}.
    \]\par
In fact, from the definition of the mixed-norm Lebesgue space, we know that $L^{\Vec{p}}(\mathbb R^{n})$ with $\Vec{p}\in(0,\infty)^{n}$ is a ball quasi-Banach function space. By \cite{huang2019}, Assumption~\ref{ass2.7} holds true when $s\in(0,1]$, $\theta\in(0,\min\{s,p_{-}\})$ and $X:=L^{\Vec{p}}(\mathbb R^{n})$ with $\Vec{p}\in(0,\infty)^{n}$. Besides, by \cite{huang2019}, the hypothesis of Lemma~\ref{le2.8} holds true when $X:=L^{\Vec{p}}(\mathbb R^{n})$ with $\Vec{p}\in(0,\infty)^{n}$, $p_{0}\in(0,p_{-})$ and $q_{0}\in(\max\{1,p_{+}\},\infty]$. Furthermore, $L^{\Vec{p}}(\mathbb R^{n})$ has an absolutely continuous quasi-norm. Therefore, Theorems~\ref{th4.1},\ \ref{th4.2},\ \ref{th4.5},\ \ref{th5.2} and Corollary~\ref{cor4.3} hold true with X replaced by $L^{\Vec{p}}(\mathbb R^{n})$.

\subsection{Orlicz-slice space}
A function $\Phi\colon[0,\infty)\to[0,\infty)$ is called an \emph{Orlicz function} if it is non-decreasing and satisfies $\Phi(0)=0$, $\Phi(t)>0$ whenever $t\in(0,\infty)$, and $\lim_{t\to\infty}\Phi(t)=\infty$. An Orlicz function $\Phi$ is said to be of \emph{lower}(resp., \emph{upper}) \emph{type} $p$ with $p\in(-\infty,\infty)$ if there exists a positive constant $C_{(p)}$, depending on $p$, such that, for any $t\in[0,\infty)$ and $s\in(0,1)$ (resp., $s\in[1,\infty)$),
    \[
    \Phi(st)\leq C_{(p)}s^{p}\Phi(t).
    \]
    An Orlicz function $\Phi\colon[0,\infty)\to[0,\infty)$ is said to be of \emph{positive lower} (resp., \emph{upper}) type if it is of lower (resp., upper) type $p$ for some $p\in(0,\infty)$.\par
    Let $\Phi$ be an Orlicz function with positive lower type $p_{\Phi}^{-}$ and positive upper type $p_{\Phi}^{+}$. The \emph{Orlicz space} $L^{\Phi}(\mathbb R^{n})$ is defined to be the set of all measurable functions $f$ on $\mathbb R^{n}$ such that
    \[
    \|f\|_{L^{\Phi}}:=\inf\left\{  \lambda\in(0,\infty)\colon\int_{\mathbb R^{n}}\Phi\left( \frac{\vert f(x)\vert}{\lambda} \right)dx\leq1 \right\}<\infty.
    \]
\begin{definition}
    Let $t,\ r\in(0,\infty)$ and $\Phi$ be an Orlicz function with lower type $p^{-}_{\Phi}\in(0,\infty)$ and upper type $p^{+}_{\Phi}\in(0,\infty)$. Then \emph{Orlicz-slice space} $(E^{r}_{\Phi})_{t}(\mathbb R^{n})$ is defined to be the set of all measurable functions $f$ such that
    \[
    \|f\|_{(E^{r}_{\Phi})_{t}}:=\left\{  \int_{\mathbb R^{n}}\left[ \frac{\|f\chi_{B(x,t)}\|_{L^{\Phi}}}{\|\chi_{B(x,t)}\|_{L^{\Phi}}}  \right]^{r} dx\right\}^{1/r}<\infty.
    \]
\end{definition}
    In fact, from \cite[Lemma 2.28]{zhang2019real} we know that $(E^{r}_{\Phi})_{t}(\mathbb R^{n})$ is a ball quasi-Banach function space. Similarly to the proof of \cite[Lemma 4.5]{zhang2019real}, we know that $(E^{r}_{\Phi})_{t}(\mathbb R^{n})$ has an absolutely continuous quasi-norm. 
    Let $X:=(E^{r}_{\Phi})_{t}(\mathbb R^{n})$ with $t,r\in(0,\infty)$ and $p^{-}_{\Phi},p^{+}_{\Phi}\in(0,\infty)$. From \cite[Lemma 4.3]{zhang2019real}, Assumption~\ref{ass2.7} holds true when $s\in(0,1]$ and $\theta\in(0,\min\{s,p_{\Phi}^{-},r\})$. Besides, 
    from \cite[Lemma 4.4]{zhang2019real}, the hypothesis of Lemma~\ref{le2.8} holds true when $q_{0}\in(\max\{1,r,p_{\Phi}^{+}\},\infty]$ and $p_{0}\in(0,\min\{p_{\Phi}^{-},r\})$. Therefore, Theorems~\ref{th4.1},\ \ref{th4.2},\ \ref{th4.5},\ \ref{th5.2} and Corollary~\ref{cor4.3} hold true with X replaced by $(E^{r}_{\Phi})_{t}(\mathbb R^{n})$.

\section*{Acknowledgments}
This project is supported by the National Natural Science Foundation of China (Grant No. 11901309), the China Postdoctoral Science Foundation (Grant No. 2023T160296) and the Jiangsu Government Scholarship for Study Abroad.


\bibliographystyle{amsplain}

\end{document}